\documentclass{amsart}

\usepackage{graphicx}
\usepackage{amssymb,amsmath,color}
\usepackage{stmaryrd}
\usepackage{dsfont}
\usepackage{mathrsfs}
\usepackage{enumitem}
\newcommand{\RR}{\mathbb{R}}

\renewcommand{\S}{\mathcal{S}}
\usepackage{color}
\usepackage{tikz}
\usepackage{pgf}
\usepackage{genyoungtabtikz}
\usepackage{bm}

\newtheorem{theorem}{Theorem}[section]

\newtheorem{lemma}[theorem]{Lemma}
\newtheorem{corollary}[theorem]{Corollary}
\newtheorem{proposition}[theorem]{Proposition}
\theoremstyle{definition}
\newtheorem{definition}[theorem]{Definition}
\newtheorem{example}[theorem]{Example}

\theoremstyle{remark}
\newtheorem{remark}[theorem]{Remark}

\newcommand{\Hzero}[0]{
\resizebox{0.58cm}{0.35cm}{
\begin{tikzpicture}
\node[fill=black, circle, minimum size=0.3cm] (1) {};
\node[fill=black, circle, minimum size=0.3cm] (2) [below left of=1] {};
\node[fill=black, circle, minimum size=0.3cm] (3) [below right of=1] {};
\end{tikzpicture}}}

\newcommand{\Hone}[0]{
\resizebox{0.58cm}{0.35cm}{
\begin{tikzpicture}
\node[fill=black, circle, minimum size=0.3cm] (1) {};
\node[fill=black, circle, minimum size=0.3cm] (2) [below left of=1] {};
\node[fill=black, circle, minimum size=0.3cm] (3) [below right of=1] {};
\draw (2)--(1);
\end{tikzpicture}}}

\newcommand{\Htwo}[0]{
\resizebox{0.58cm}{0.35cm}{
\begin{tikzpicture}
\node[fill=black, circle, minimum size=0.3cm] (1) {};
\node[fill=black, circle, minimum size=0.3cm] (2) [below left of=1] {};
\node[fill=black, circle, minimum size=0.3cm] (3) [below right of=1] {};
\draw (2)--(1)--(3);
\end{tikzpicture}}}

\newcommand{\Hthree}[0]{
\resizebox{0.58cm}{0.35cm}{
\begin{tikzpicture}
\node[fill=black, circle, minimum size=0.3cm] (1) {};
\node[fill=black, circle, minimum size=0.3cm] (2) [below left of=1] {};
\node[fill=black, circle, minimum size=0.3cm] (3) [below right of=1] {};
\draw (2)--(1)--(3)--(2);
\end{tikzpicture}}}

\newcommand{\labeledcherries}[3]{
\resizebox{0.8cm}{0.4cm}{
\begin{tikzpicture}
\node[fill=black, circle, minimum size=0.3cm, label=right:\Huge{\textbf{#1}}] (1) {};
\node[fill=black, circle, minimum size=0.3cm, label=left:\Huge{\textbf{#2}}] (2) [below left of=1] {};
\node[fill=black, circle, minimum size=0.3cm, label=right:\Huge{\textbf{#3}}] (3) [below right of=1] {};
\draw (2)--(1)--(3);
\end{tikzpicture}}}

\newcommand{\flagnotflag}[4]{
\resizebox{0.8cm}{0.6cm}{
\begin{tikzpicture}
\node[fill=black, circle, minimum size=0.3cm, label=right:\Huge{\textbf{#1}}] (1) {};
\node[fill=black, circle, minimum size=0.3cm, label=left:\Huge{\textbf{#2}}] (2) [below left of=1] {};
\node[fill=black, circle, minimum size=0.3cm, label=right:\Huge{\textbf{#3}}] (3) [below right of=1] {};
\node[fill=black, circle, minimum size=0.3cm, label=right:\Huge{\textbf{#4}}] (4) [below of=3] {};
\draw (2)--(1)--(3)--(4);
\end{tikzpicture}}}

\newcommand{\qexample}[4]{
\resizebox{0.8cm}{0.6cm}{
\begin{tikzpicture}
\node[fill=black, circle, minimum size=0.3cm, label=left:\Huge{\textbf{#1}}] (1) {};
\node[fill=black, circle, minimum size=0.3cm, label=left:\Huge{\textbf{#2}}] (2) [below of=1] {};
\node[fill=black, circle, minimum size=0.3cm, label=right:\Huge{\textbf{#3}}] (3) [right of=2] {};
\node[fill=black, circle, minimum size=0.3cm, label=right:\Huge{\textbf{#4}}] (4) [right of=1] {};
\draw (1)--(2)--(3)--(4)--(1);
\end{tikzpicture}}}

\newcommand{\notflag}[4]{
\resizebox{0.8cm}{0.6cm}{
\begin{tikzpicture}
\node[fill=black, circle, minimum size=0.3cm, label=right:\Huge{\textbf{#1}}] (1) {};
\node[fill=black, circle, minimum size=0.3cm, label=left:\Huge{\textbf{#2}}] (2) [below left of=1] {};
\node[fill=black, circle, minimum size=0.3cm, label=right:\Huge{\textbf{#3}}] (3) [below right of=1] {};
\node[fill=black, circle, minimum size=0.3cm, label=right:\Huge{\textbf{#4}}] (4) [below of=3] {};
\draw (3)--(2)--(1)--(3)--(4);
\end{tikzpicture}}}

\newcommand{\labeledclaw}[4]{
\resizebox{0.8cm}{0.6cm}{
\begin{tikzpicture}
\node[fill=black, circle, minimum size=0.3cm, label=right:\Huge{\textbf{#1}}] (1) {};
\node[fill=black, circle, minimum size=0.3cm, label=left:\Huge{\textbf{#2}}] (2) [below left of=1] {};
\node[fill=black, circle, minimum size=0.3cm, label=below:\Huge{\textbf{#3}}] (3) [below of=1] {};
\node[fill=black, circle, minimum size=0.3cm, label=right:\Huge{\textbf{#4}}] (4) [below right of=1] {};
\draw (2)--(1)--(3)--(1)--(4);
\end{tikzpicture}}}

\newcommand{\labeledvedge}[2]{
\resizebox{0.32cm}{0.4cm}{
\begin{tikzpicture}
\node[fill=black, circle, minimum size=0.3cm, label=left:\Huge{\textbf{#1}}] (1) {};
\node[fill=black, circle, minimum size=0.3cm, label=left:\Huge{\textbf{#2}}] (2) [below of=1] {};
\draw (2)--(1);
\end{tikzpicture}}}

\newcommand{\labeledvnonedge}[2]{
\resizebox{0.32cm}{0.4cm}{
\begin{tikzpicture}
\node[fill=black, circle, minimum size=0.3cm, label=left:\Huge{\textbf{#1}}] (1) {};
\node[fill=black, circle, minimum size=0.3cm, label=left:\Huge{\textbf{#2}}] (2) [below of=1] {};
\end{tikzpicture}}}

\begin{document}

\title{Symmetry in Tur{\'a}n Sums of Squares Polynomials from Flag Algebras}

\author{Annie Raymond}
\address{Department of Mathematics, University of Washington, Box
  354350, Seattle, WA 98195, USA} \email{raymonda@uw.edu}

\author{Mohit Singh}
\address{H. Milton Stewart School of Industrial and Systems Engineering, Georgia Institute of Technology, Atlanta, GA 30332}
\email{mohitsinghr@gmail.com}

\author{Rekha R. Thomas}
\address{Department of Mathematics, University of Washington, Box
  354350, Seattle, WA 98195, USA} \email{rrthomas@uw.edu}

\date{\today}

\maketitle

\begin{abstract} Tur\'an problems in extremal combinatorics
ask to find asymptotic bounds on the edge densities of
graphs and hypergraphs that avoid specified subgraphs.
The theory of flag algebras proposed by Razborov provides powerful
methods based on semidefinite programming to find sums of squares that establish
edge density inequalities in Tur\'an problems.
Working with polynomial analogs of
the flag algebra entities, we prove that such sums of squares created by flag algebras
can be retrieved
from a restricted version of the symmetry-adapted semidefinite program
proposed by Gatermann and Parrilo. This involves using the representation theory of the symmetric group for finding succinct sums of squares expressions for invariant polynomials.
The connection reveals
several combinatorial and structural properties of flag algebra sums of squares, and offers
new tools for Tur\'an and other related problems.
\end{abstract}

\section{Introduction}
The \emph{Tur{\'a}n problem}  from extremal combinatorics asks the following question: given a graph $A$, what is the maximum number of edges in a graph on $n$ vertices not containing $A$ as a subgraph? Tur{\'a}n \cite{Turan} answered this question for $A=K_s$, the complete graph on $s$ vertices, generalizing a classical result of Mantel~\cite{Mantel07} for triangle-free graphs, and
establishing the field of extremal graph theory. In general, for any graph $A$, Erd\"os and Stone~\cite{ES46} identified the maximum possible density of edges in any $A$-free graph asymptotically. The \emph{hypergraph Tur{\'a}n problem} asks the same question for hypergraphs, but the current understanding of this problem is far from satisfactory. In particular, even asymptotically, tight bounds on the maximum number of edges in a $n$-vertex $3$-uniform hypergraph\footnote{A $r$-uniform hypergraph has (hyper)edges of size $r$.} not containing a complete graph of size four is not known. A variety of general techniques have been developed to prove bounds for this long-standing hypergraph Tur{\'a}n problem; see for example \cite{ChungLu}, \cite{FranklFuredi}, \cite{LuZhao}, \cite{Pikhurko}, \cite{Sidorenko}, and \cite{Keevash11} for a survey.

Recently, semidefinite programming methods arising from the powerful theory of \emph{flag algebras} introduced by Razborov~\cite{RazborovFlagAlgebras} have led to significant progress on this problem. Indeed, many of the previous bounds can be proven via this technique and several new results giving the tightest known bounds have been obtained~\cite{Razborov10,Razborov13,Razborov14,Falgas-RavryVaughan}. For instance, Razborov in \cite{Razborov10} proved that the (maximum) asymptotic edge density of a $3$-uniform hypergraph without a $4$-clique is 0.561666 (Tur\'an~\cite{Turan61} conjectured it to be $\frac{5}{9}$). Moreover, he also showed in the same paper that, if one forbids an additional subgraph, then the asymptotic edge density is indeed $\frac{5}{9}$.  Razborov's method relies on establishing inequalities involving densities of suitably chosen subgraphs in any $n$-vertex graph/hypergraph. This is done by lower bounding density expressions with a scalar sum of squares (sos) coming from flags. A suitable sos is found by formulating a semidefinite program (SDP) whose size depends on the flags that are used. The key to the success of this method is that the size of the SDP is thus independent of the number of vertices, which is particularly helpful for asymptotic results. However, deciding which flags are needed to construct the flag sos expressions is an art.

Our work is motivated by the basic question as to whether there is a fundamental connection between Razborov's scalar flag sos methods and the more standard sos theory for polynomials.  Expressing a polynomial with real coefficients as a sos of polynomials in order to certify its nonnegativity is a well-established technique in real algebraic geometry going back at least to the 19th century. In recent years, these ideas have acquired new life following the realization that sos polynomials can be found via the modern tool of semidefinite programming which has led to remarkable progress in optimization and algorithm design. For  an introduction to these methods, see one of \cite[Chapters 1 \& 2]{BPTSIAMBook}, \cite{LaurentSurvey}, or \cite{Parrilo}.

In this paper, we show that indeed there is a deep connection between the sos methods coming from flags and those for polynomials in real algebraic geometry. We show that symmetry-reduction in polynomial optimization is precisely the right framework through which this relationship can be established. This brings in tools from the representation theory of the symmetric group, highlighting the many combinatorial features of flag sos expressions.

Symmetry-reduction in polynomial optimization, or more generally semidefinite programming, is a powerful technique and has been useful in many settings \cite {Schrijver79}, \cite{SchrijverInvariantSDPSurvey}, \cite{associationSchemesSurvey}, \cite{GatermannParrilo}. When a nonnegative polynomial is invariant under the action of a finite group, the representation theory of the group can be used to simplify the SDP used to obtain the sos certificate for its nonnegativity.
In \cite{GatermannParrilo}, Gatermann and Parrilo show that in this invariant setting the
original SDP breaks into several smaller (but coupled) SDPs, each indexed by an \emph{irreducible representation} of the group, leading to tremendous computational savings. We appeal to this framework to establish our results.

Our main technical result shows that the flag algebra method for establishing graph density inequalities embeds naturally in a restricted version of the Gatermann-Parrilo symmetry-adapted SDP.  Our results rely on the rich combinatorics hidden in the sos expressions coming from flag algebras that we expose using the representation theory of the symmetric group.
We give a precise description of the symmetry-reduced SDP in terms of the irreducible representations of the symmetric group and show that only certain irreducibles are needed.
Consequently, we prove that the size of this SDP is independent of the number of vertices in the associated graphs, as in Razborov's methods. This offers a systematic way of establishing graph density inequalities, and more generally, Cauchy-Schwarz proofs coming from flags, through standard sos methods where no sophisticated choices are necessary.  

\subsection{Our results in detail and the organization of the paper}
For notational simplicity, we restrict the exposition in this paper to Tur\'an-type problems in the setting of graphs. Our results extend naturally to the broader realm of hypergraphs, digraphs, tournaments, etc, which carry many open problems.
In the conclusion of this paper, we will elaborate on the modifications
needed for these extensions.

Restricting to the setting of graphs,
assume that there is only one graph $A$ that must be avoided in a Tur\'an problem.  The more general setting of avoiding all graphs in a family is treated similarly. Also, we will work in the setting of avoiding $A$ as an {\em induced} subgraph. Note that this setting is general, and that the non-induced setting can be modeled through it by forbidding every induced graph containing $A$.  The main technical challenge in Tur\'an problems is to show an upper bound on the edge density of any $A$-free graph.

 The first step in linking flag algebra methods for Tur\'an problems to the symmetry-reduction techniques of
\cite{GatermannParrilo} is to view graph density expressions as polynomials modulo an ideal, and flag sos expressions as polynomial sos expressions.  This is done in Section~\ref{sec:FlagAlgebras}.

We begin with a few basic definitions. For a fixed positive integer $n$, the polynomials we work with
lie in $\RR[{\bf x}] := \RR[\mathsf{x}_{ij} \,:\, 1 \leq i < j \leq n]$, the polynomial ring over $\RR$ in $n \choose 2$ variables indexed by the edges in the complete graph $K_n$.
If $A$ denotes the graph that must be avoided in the Tur\'an problem we are interested in, then the ideal we need, $\mathscr{I}_n^{{A}}$, is precisely the set of polynomials in $\RR[{\bf x}]$ that vanish on the
characteristic vectors of all graphs on $n$ vertices that do not contain ${A}$ as an induced subgraph.
A polynomial is nonnegative on this finite set of characteristic vectors if and only if it is equivalent to a sos polynomial {\em modulo} $\mathscr{I}_n^{{A}}$.

Two polynomials $\mathsf{f}$ and $\mathsf{g}$ are equivalent modulo
$\mathscr{I}_n^{{A}}$ if and only if
$\mathsf{f}-\mathsf{g} \in \mathscr{I}_n^{{A}}$,
written as $\mathsf{f} \equiv \mathsf{g}$ mod $\mathscr{I}_n^{{A}}$.
This equivalence relation differentiates between functions on the zeros of $\mathscr{I}_n^{{A}}$,
i.e., $\mathsf{f} \equiv \mathsf{g}$ mod $\mathscr{I}_n^{{A}}$ if and only if
$\mathsf{f}(\bm{v}) = \mathsf{g}(\bm{v})$ for every zero $\bm{v}$ of
$\mathscr{I}_n^{{A}}$. Therefore, if
$\mathsf{g} = \sum \mathsf{h}_j^2$ is a sos, then $\mathsf{f}$ is nonnegative on the characteristic vectors of all
${A}$-free graphs on $n$ vertices.
Further, we say that $\mathsf{f}$ is $d$-sos mod $\mathscr{I}_n^{{A}}$ if each $\mathsf{h}_j$ has degree at most $d$. 
Every $d$-sos polynomial has the form $\mathbf{y}^T Q \mathbf{y}$  where $Q$ is a $l\times l$ positive semidefinite (psd) matrix and $\mathbf{y}$ is a $l$-dimensional vector whose coordinates are polynomials of degree at most $d$.
This allows us to use semidefinite programming to
search for $d$-sos expressions for a given polynomial modulo $\mathscr{I}_n^{{A}}$.  

In Section~\ref{sec:FlagAlgebras}, we describe the polynomial analogs of the ingredients in a sos proof coming from flag algebras. This allows us to translate flag sos expressions to polynomial sos expressions for density polynomials as shown in Propositions  \ref{prop:translation-sos} and \ref{cor:translation-sos}.
We illustrate our polynomial translation on Mantel's theorem. The complete proof, including the flag algebra approach for Tur\'an problems, can be found in the Appendix.

Given the polynomial formulation, an approach to showing an upper bound on the graph density polynomial is to use the standard sos method in polynomial optimization.  This raises several natural questions. Firstly, whether flag sos expressions can be retrieved via this approach? Secondly, whether the size of the SDP formulated in this approach will be independent of $n$ as in the flag algebra framework? In this paper, we answer both questions affirmatively.

A priori, searching for a $d$-sos proof leads to a SDP formulation whose size grows with $n$. The first step in establishing our result is to notice that the graph density polynomial whose nonnegativity we need to establish is invariant under an action induced by the symmetric group $\mathfrak{S}_n$ on $n$ letters acting on the vertices of $K_n$. Therefore, one can use the symmetry-reduction techniques in \cite{GatermannParrilo} to simplify the computational cost of searching for its sos certificate.
This in turn relies on the representation theory of $\mathfrak{S}_n$; we explain the basics of this theory in Section~\ref{sec:representation theory basics}. We then describe the symmetry-reduction strategy of \cite{GatermannParrilo} in Section~\ref{sec:GP method} which breaks the SDP that searches for a sos expression into smaller SDPs that are indexed by the irreducible representations of $\mathfrak{S}_n$ or, equivalently, the partitions of $n$.
This section is largely expository but it is crucial for understanding our main results in Section~\ref{sec:main results}. For efficiency, we tailor all discussion of \cite{GatermannParrilo}
to $\mathfrak{S}_n$ which in turn creates a new set of combinatorial tools for problems to which flag sos methods apply.

In Section~\ref{sec:main results}, we come to our main results. Suppose we fix a maximum degree $d$ for the sos polynomials we are searching for, and let $\RR[{\bf x}]_{\leq d}$ denote the vector space of all polynomials in $\RR[{\bf x}]$ of degree up to $d$. The group $\mathfrak{S}_n$ breaks $\RR[{\bf x}]_{\leq d}$ into a direct sum of subspaces indexed by the
partitions of $n$, called the isotypic decomposition of $\RR[{\bf x}]_{\leq d}$. We first establish that the atomic pieces of the sos polynomials that come out of flag algebras
are invariant with respect to the {\em row group} of a {\em tableau} defined from the  flags that are chosen  by the flag algebra method (Theorem~\ref{lcorbits}). Next, using the theory of restricted representations, we decide which subset of subspaces in the isotypic decomposition of $\RR[{\bf x}]_{\leq d}$ contain the polynomials in a flag algebra sos in their span. Finally, we show that
the 
sos expressions from flag algebras can be retrieved from the Gatermann-Parrilo SDP restricted to those partitions that survive in the previous step.
The key point to note here is that the number of partitions indexing these necessary subspaces is not a function of $n$, but rather of $d$. So if we fix $d$ and let $n$ go to infinity, the number of partitions, and the sizes of the corresponding SDPs, will stay fixed. This answers the two questions raised above and links flag algebra methods to symmetry-reduction in semidefinite programming. 
Under the usual action of 
$\mathfrak{S}_n$ on the polynomial ring $\RR[x_1, \ldots, x_n]$, \cite{Thorsten} has also shown that one can get sos expressions for symmetric polynomials whose size is independent of $n$.

In Section~\ref{sec:Mantel1}, we illustrate our main results on the Mantel example. We will see that for any $n$,
just two specific partitions are enough to obtain the sos expression from flag algebras.
Since there are many simple proofs of Mantel's result, the goal of using this example is simply to
illustrate the chain of results that make up this paper. It is both simple and rich enough for this purpose.

We conclude in Section~\ref{sec:conclusion} where we explain how to apply our techniques in other settings such as directed graphs, hypergraphs and tournaments, and also discuss different future directions for this work.

\vspace{0.1cm}

\noindent{\bf Acknowledgements}. We thank Greg Blekherman for several important inputs to this paper relating to representation theory. They came at crucial junctures and helped us along greatly. We also thank  Andrew Berget, Monty McGovern, Pablo Parrilo, James Pfeiffer, Paul Smith and Vasu Tewari for helpful conversations and suggestions.
We also thank the referees of this paper for their valuable comments which have improved the content and exposition.

\newcommand{\pd}{\mathsf{d}}
\newcommand{\pp}{\mathsf{p}}
\newcommand{\EE}{\mathbb{E}}

\section{Sums of Squares from Flag Algebras}  \label{sec:FlagAlgebras}
In this section we present the polynomial analogs of
Razborov's flag algebra method for proving
inequalities on graph densities. As mentioned in the Introduction, we restrict our attention to graphs. For generalizations, see the conclusion of the paper. The key new feature distinguishing this work from the literature (\cite{Falgas-RavryVaughan}, \cite{RazborovFlagAlgebras}, \cite{Razborov14}) is that we think of all densities as polynomials that can be evaluated on characteristic vectors of graphs. If the reader is unfamiliar with flag algebras, we highly recommend reading Section~2.2 of \cite{Falgas-RavryVaughan} first to see a concrete application of the flag method to Mantel's theorem. The polynomial version may appear more difficult to parse at first but they are simply functions on the (finite) set of characteristic vectors of the graphs allowed by the problem that evaluate to density expressions as in \cite{Falgas-RavryVaughan} and \cite{Razborov10}.  

The polynomial translation illustrates how certificates based on flag algebras can be interpreted as polynomial sos proofs. The polynomials appearing in Razborov's sos proofs and their specific symmetries will be key in obtaining our main results efficiently via the Gatermann-Parrilo framework introduced in Section~\ref{sec:GatermannParrilo}. To illustrate the method on an example, we use the polynomial version of flag algebras to prove Mantel's theorem at the end of this section.

Consider the general problem of certifying an inequality involving graph densities over all graphs with a
certain property. We first fix $n$, the number of vertices in our graphs. Let $\mathcal{G}$ be the set of all undirected graphs up to isomorphism with the desired property, and let $\mathcal{G}_n$ be the graphs in $\mathcal{G}$ that have $n$ vertices.
Also, let $V(G)$ and $E(G)$ denote the sets of vertices and edges of $G$. We represent a graph $G$ by its characteristic vector  $\mathds{1}_G \in \{0,1\}^{\binom{n}{2}}$ whose $ij$-th coordinate is $1$ if and only if $\{i,j\}\in E(G)$. Throughout,
we work with the polynomial ring $\RR[{\mathbf x}] = \RR[\mathsf{x}_{ij}:1 \leq i < j \leq n]$ described earlier.

Fix an integer $m$ such that $m<n$, and let $H \in \mathcal{G}_m$. Furthermore, let $\textup{Inj}(S,[n])$ denote the set of injective maps $\alpha:S\rightarrow [n]$ for any set $S$. For a fixed $\alpha \in
\textup{Inj}(V(H),[n])$, define the polynomial
\begin{align} \label{eq:qhf}
\mathsf{q}^\alpha_H := \prod_{\{i,j\}\in E(H)}\mathsf{x}_{\alpha(i)\alpha(j)}\prod_{\{i,j\}\in \binom{V(H)}{2}\backslash E(H)}(1-\mathsf{x}_{\alpha(i)\alpha(j)}) \in \RR[{\bf x}].
\end{align}
For a graph $G \in \mathcal{G}_n$,  $\mathsf{q}^\alpha_H(\mathds{1}_G) = 0$ if and only if at least one of the following is true: (i) for some $\{i,j\} \in E(H)$, $\{\alpha(i), \alpha(j) \}$ is not an edge of $G$ or (ii) for some $\{i,j\} \in \binom{V(H)}{2}\backslash E(H)$, $\{\alpha(i),\alpha(j)\}$ is an edge of $G$. If $\mathsf{q}_H^\alpha(\mathds{1}_G)=1$, we say that the vertices $\alpha(V(H))$ \emph{label-induce} $H$ in $G$. Being label-induced is stronger than being induced, since the set of vertices $\alpha(V(H)$ might induce (in the usual sense of the word) $H$ in $G$ even if $\mathsf{q}_H^\alpha(\mathds{1}_G)=0$. If $\alpha(V(H))$ does not induce $H$ in $G$, then $\mathsf{q}_H^\alpha(\mathds{1}_G)=0$.
\begin{example}
Let $H=$\qexample{0}{1}{2}{3}. Suppose that $\alpha(i)=i$ for every $i\in \{0,1,2,3\}$. Then $q_H^\alpha(\mathds{1}_H)=1$, i.e., $\alpha(V(H))$ label-induces $H$ in $H$. However, if $\alpha(0)=0$, $\alpha(1)=1$, $\alpha(2)=3$ and $\alpha(3)=2$, then $q^\alpha_H(\mathds{1}_H)=0$ (since, for example, $\alpha(1)$ and $\alpha(2)$ do not form an edge) even though $\alpha(V(H))$ does induce $H$ in $H$.
\end{example}

We can use the polynomials $\mathsf{q}^\alpha_H$ for different injections $\alpha$ to calculate the density of $H$ in $G$. Define
\begin{align} \label{eq:dh}
\mathsf{p}_H & :=\frac{1}{\binom{n}{|V(H)|}} \cdot \sum_{\substack{S\subseteq [n] :\\|S|=|V(H)|}} \frac{1}{a_H} \sum_{\alpha\in \textup{Inj}(V(H),S)} \mathsf{q}_H^\alpha \\
& = \frac{1}{a_H \binom{n}{|V(H)|}} \cdot \sum_{\alpha\in \textup{Inj}(V(H),[n])} \mathsf{q}_H^\alpha
\end{align}

where

$$a_H = \sum_{\alpha\in \textup{Inj}(V(H),V(H))} \mathsf{q}_H^\alpha(\mathds{1}_H).$$
The quantity $a_H$ is the number of label-induced copies of $H$ in itself. Thus, for any subset $S \subseteq[n]$ of size $|V(H)|$, if $S$ induces $H$ in $G$, then $$\sum_{\alpha\in \textup{Inj}(V(H),S)} q_H^\alpha(\mathds{1}_G) = a_H.$$ Otherwise, $\displaystyle\sum_{\alpha\in \textup{Inj}(V(H),S)} q_H^\alpha(\mathds{1}_G) =0$. Therefore, $\frac{1}{a_H}\displaystyle\sum_{\alpha\in \textup{Inj}(V(H),S)} q_H^\alpha(\mathds{1}_G) \in \{0,1\}$ and it is one if and only if $S$ induces $H$ in $G$.

\begin{example}
Again, let $H=$\qexample{0}{1}{2}{3}. Out of the $4!=24$ maps in $\textup{Inj}(V(H),V(H))$, eight of them label-induce $H$ in itself, namely the maps $\alpha_j$ and $\beta_j$ for $j\in \{0,1,2,3\}$ such that $\alpha_j(i)\equiv i+j \mod 4$ and $\beta_j(i)\equiv 4-i+j \mod 4$ for every $i \in \{0,1,2,3\}$. Thus $a_H=8$.
\end{example}

Hence, evaluated on $\mathds{1}_G$, $\mathsf{p}_H$ yields the density of $H$ in $G$,
i.e., the probability that a collection of $m$ vertices in $G$ chosen uniformly at
random will induce a copy of $H$ in $G$ up to isomorphism. Note that $\mathsf{p}_H(\mathds{1}_G)=p(H,G)$ in \cite{RazborovFlagAlgebras}.

A {\em type} of size $k$ ($\leq n$) is a $k$-vertex graph $\sigma$
in which every vertex is labeled with a distinct element of
$[k]$. For an integer $l \geq k$, a $\sigma$-{\em flag} $F$ of size $l$ is a graph in $\mathcal{G}_l$ which has $k$
vertices labeled $1, \ldots, k$ and the bijective map that sends vertex labeled $i$ in $\sigma$ to the vertex
labeled $i$ in $F$
label-induces
a copy of $\sigma$ in $F$ with identical labels for the vertices.  Let
$\mathcal{F}^\sigma_l$ be the set of all $\sigma$-flags of size $l$ up
to isomorphism.

\begin{example}
Let $\sigma=\labeledcherries{1}{2}{3}$. Then, \flagnotflag{1}{2}{3}{} is in $\mathcal{F}^\sigma_4$. However, \flagnotflag{2}{1}{3}{} is not since $\sigma$ is mislabeled; \notflag{1}{2}{3}{} also is not since vertices $1,2,3$ do not induce $\sigma$.
\end{example}

Fix a type $\sigma$ of size $k$ and $l\geq k$.  For a flag
$F \in \mathcal{F}^\sigma_l$ and an injective map $\alpha:V(F)\rightarrow [n]$, we have  the
polynomial $\mathsf{q}^\alpha_F$ as in \eqref{eq:qhf}. Now suppose $\theta$ is a fixed labeling of $k$ vertices in $G$ using all the labels in $[k]$, i.e., $\theta\in \textup{Inj}([k],[n])$ is an injective map from $[k]$ to $[n]=V(G)$. We say $\alpha:V(F)\rightarrow [n]$ respects the labeling $\theta$ if $\alpha(v)=\theta(i)$ for any vertex $v\in V(F)$ labeled $i\in [k]$. Let $\textup{Inj}_\theta(V(F),[n])$ denote the set of injective maps $\alpha:V(F)\rightarrow [n]$ that respect the labeling $\theta$.  Then we define the following density polynomial
\begin{align}\label{dthetaf}
\mathsf{p}^\theta_F&:=\frac{1}{\binom{n-k}{l-k}} \sum_{\substack{S\subseteq [n]: \\|S|=|V(F)|,\\ S\supseteq \textup{im}(\theta)}} \frac{1}{a_F^\sigma} \sum_{\alpha\in \textup{Inj}_\theta(V(F),S)} \mathsf{q}_F^\alpha\\
&=\frac{1}{a_F^\sigma \binom{n-k}{l-k}} \sum_{\alpha\in \textup{Inj}_\theta(V(F),[n])} \mathsf{q}_F^\alpha,
\end{align}

where

$$a_F^\sigma = \sum_{\substack{\alpha\in \textup{Inj}(V(F),V(F)):\\ \alpha(v)=v \ \forall v\in V(\sigma)}} \mathsf{q}_F^\alpha(\mathds{1}_F).$$
The quantity $a_F^\sigma$ is the number of label-induced copies of $F$ in $F$ such that every 
vertex in $\sigma$ is sent to itself.
Note that $\mathsf{p}_F^\theta(\mathds{1}_G)=p(F_1;F_2)$ in \cite{RazborovFlagAlgebras} where $F_1=(F,\theta)\in \mathcal{F}_l^\sigma$ and $F_2=(G,\theta)\in \mathcal{F}_n^\sigma$ where $|\sigma|=k$. Indeed, $\mathsf{p}^\theta_F(\mathds{1}_G)$ is the probability that the $k$ vertices of $G$ labeled by $\theta$ along with the remaining  $l-k$ unlabeled vertices picked uniformly at random induce a copy of $F$ in $G$.

Razborov's flag algebra methods can be used to certify the nonnegativity of graph density functions, i.e, functions involving $\mathsf{p}_H(\mathds{1}_G)$ (or Razborov's $p(H,G)$) for different graphs $H$. For example, to retrieve Mantel's theorem, we want to show that $r=\frac{1}{2}-p(e,G)$ is nonnegative over all triangle-free graphs on $n$ vertices where $n\rightarrow \infty$. Here, $e$ is the 2-vertex graph consisting of one edge. This is done by
expressing the given graph density function $r$ as a sos of linear combinations of flag densities, thus establishing the nonnegativity of $r$. This in turn involves finding a psd matrix $M$ such that $r$ is the average over $\theta$'s of $(f_1(G), \ldots, f_s(G))M(f_1(G), \ldots, f_s(G))^\top$ where $f_i(G)$ is a linear combination of $\sigma$-flag densities $\mathsf{p}_F^\theta(\mathds{1}_G)$ for different flags $F$. Razborov refers to such certificates as \emph{Cauchy-Schwarz} proofs (for the nonnegativity of $r$).

The main result of this section is that one can interpret Razborov's Cauchy-Schwarz certificates as sum of squares of polynomials modulo an ideal, using the polynomial analogs of densities that we constructed in this section. First we give a high level version of this result.

\begin{proposition}\label{prop:translation-sos}
Let $r$ be a function of graph densities that is nonnegative over a family of graphs as the the number of vertices goes to infinity, and let $\mathsf{r}$ be the polynomial analog of $r$. Suppose we are given a Cauchy-Schwarz proof of the nonnegativity of $r$ where the  flags used have type $\sigma$ and size at most $l$.  Then $\mathsf{r}$ can be written as a sum of squares of polynomials modulo the vanishing ideal of the family of graphs under consideration.
\end{proposition}

\proof
Suppose $r$ has a certificate of nonnegativity in a Cauchy-Schwarz proof using flags of size $l$ and
type $\sigma$.  Then this certificate is the average of an expression of the form
$$(f_1(G), \ldots, f_s(G)) M (f_1,(G) \ldots, f_s(G))^\top$$ over all $\theta$, where each $f_i(G)$ is a linear combination of $\sigma$-flags densities and $M$ is psd. In Razborov's language this means that $f_i = \sum_j b_j^i p(B_j, B)$ where $B_j = (F_j, \theta)\in \mathcal{F}_l^\sigma$ and $B=(G,\theta)\in \mathcal{F}_n^\sigma$.

We now replace these expressions by their polynomial analogs.
Let $\mathscr{I}$ be the vanishing ideal of the characteristic vectors all graphs in the family being considered.
Since $f_i(G) = \sum_{j} b_j^i \mathsf{p}_{F_j}^\theta(\mathds{1}_G)$, for a graph $G$ in our family,
$f_i(G)$ is the evaluation of the polynomial
$\mathsf{f}_i := \sum_{j} b_j \mathsf{p}_{F_j}^\theta$ on the characteristic vector of $G$ which is a zero of $\mathscr{I}$. The averaging
in the sos expression for $r$ amounts to averaging $(\mathsf{f}_1, \ldots, \mathsf{f}_s)M(\mathsf{f}_1, \ldots, \mathsf{f}_s)$ over $\theta$. This averaging in turn, is equivalent to taking the expectation over all maps $\theta$. Thus overall, we get that
$$\mathsf{r} \equiv \EE_{\theta} \left[ (\mathsf{f}_1, \ldots, \mathsf{f}_s)M(\mathsf{f}_1, \ldots, \mathsf{f}_s)^\top \right] \,\,\textup{ mod } \mathscr{I}$$
which says that when evaluated on the zeros of $\mathscr{I}$, which are precisely the characteristic vectors of graphs in our family, the function $r$ and the expression on the right hand side are equal. This is exactly what the scalar sos in the Cauchy-Schwarz proof was saying.
\qed

We now write a more precise version of Proposition~\ref{prop:translation-sos} that will be helpful in later sections.

\begin{proposition}\label{cor:translation-sos}
Assume the same hypotheses as in Proposition~\ref{prop:translation-sos}, and consider the vector
of flag density polynomials $\mathbf{p}^{\theta, \sigma, l}=(\pp^{\theta}_{F})_{F\in  \mathcal{F}^\sigma_{l}}$ for all flags with type $\sigma$, size $l$, a fixed numbering $\theta$.
Then there exists a psd matrix $Q \in \RR^{|\mathcal{F}^\sigma_{l}|\times |\mathcal{F}^\sigma_{l}|}$  and a sos certificate for the nonnegativity of $\mathsf{r}$ of the following form:
\begin{align}\label{razborovsosshape}
\mathsf{r} \equiv \EE_{\theta} \left[ {\mathbf{p}^{\theta, \sigma, l}}^\top Q \mathbf{p}^{\theta, \sigma, l} \right]
\end{align}
modulo the vanishing ideal $\mathscr{I}$ of the family of graphs under consideration.
\end{proposition}

\proof
As in the proof of Proposition~\ref{prop:translation-sos}, we arrive at the expression
$$\mathsf{r} \equiv \EE_{\theta} \left[ (\mathsf{f}_1, \ldots, \mathsf{f}_s)M(\mathsf{f}_1, \ldots, \mathsf{f}_s)^\top \right] \,\,\textup{ mod } \mathscr{I},$$
where $\mathsf{f}_i = \sum_{j} b_j \mathsf{p}_{F_j}^\theta$ and $M$ is a psd matrix.

Writing $M = {(M^{\frac{1}{2}})}^\top M^{\frac{1}{2}}$ where $M^{\frac{1}{2}} = (m_{ji})$  is a $p \times s$ matrix, we have
$$(\mathsf{f}_1, \ldots, \mathsf{f}_s)M(\mathsf{f}_1, \ldots, \mathsf{f}_s)^\top = \sum_{j=1}^p \left(\sum_{i=1}^s m_{ji} f_i\right)^2.$$

Since each $\mathsf{f}_i$ is a linear combination of $\mathsf{p}_{F_t}^\theta$'s, there exists $q_{ti} \in \mathbb{R}$ such that

\begin{align*}
(\mathsf{f}_1, \ldots, \mathsf{f}_s)M(\mathsf{f}_1, \ldots, \mathsf{f}_s)^\top &= \sum_{j=1}^p \left(\sum_{F_t \in \mathcal{F}_l^\sigma} q_{ti} \mathsf{p}_{F_t}^\theta\right)^2 ={\mathbf{p}^{\theta, \sigma, l}}^\top Q \mathbf{p}^{\theta, \sigma, l}
\end{align*}
where $\mathbf{p}^{\theta, \sigma, l}=(\pp^{\theta}_{F})_{F\in  \mathcal{F}^\sigma_{l}}$ and $Q$ is a psd matrix of size $p \times |\mathcal{F}_l^\sigma|$.
\qed

Note that the polynomials that are squared in the sos \eqref{razborovsosshape}, namely the components of the vector $Q^{\frac12} \mathbf{p}^{\theta, \sigma,l}$,
are linear combinations of the flag density polynomials  $\{\pp^{\theta}_F: F\in \mathcal{F}^\sigma_{l}\}$.
This immediately yields some quantitative bounds on the size of $Q$ and the degree of the sos.

\begin{corollary}
The degree of the above sos-proof equals the maximum degree of the polynomials of the form $\pp_{F}^{\theta}$ which is at most $\binom{l}{2}$ (for graphs) where $l$ is the number of vertices in the $\sigma$-flag. Moreover, the size of $Q$ in Proposition~\ref{cor:translation-sos} depends only on the size of $|\mathcal{F}_l^\sigma|$ and not on $n$ (when $k$ and $l$ are fixed).
\end{corollary}

We remark that certain flag Cauchy-Schwarz proofs require choosing several types $\sigma$ and sizes of flags $l$, and then taking a conic combination of sos for each $\sigma$ and $l$. This however does not change the explanations above since the argument stands for each sos.

We illustrate the above Propositions by
providing a  polynomial version of the Cauchy-Schwarz proof of Mantel's theorem presented in \cite{Falgas-RavryVaughan}.

\begin{example}
Consider the problem of finding the maximum edge density of a graph which does not contain any triangle (which we denote by $K_3$). Mantel's famed result states that the maximum edge density of a triangle-free graph
goes to $\frac{1}{2}$ as the number of vertices goes to infinity. Let $\mathcal{G}$ be the family of triangle-free graphs and $\mathcal{G}_n$ be the set of triangle-free graphs on $n$ vertices for a fixed $n$. We need to choose $\sigma$ and $l$ to obtain a flag Cauchy-Schwarz proof of the non-negativity of $\frac{1}{2}-p(e,G)$ for $G \in \mathcal{G}_n$ as $n\rightarrow \infty$.

The
characteristic vectors of $K_3$-free graphs in $\mathcal{G}_n$ are precisely the zeros of the ideal
\begin{align*}
\mathscr{I}_n^{K_3} = \langle \mathsf{x}_{ij}^2 - \mathsf{x}_{ij}, \,\,1 \leq i < j \leq n \rangle
+ \langle x_{ij}x_{ik}x_{il} \,:\, 1 \leq i < j < k \leq n \rangle.
\end{align*}
The polynomial $\mathsf{p}_n := \frac{1}{{n \choose 2}}\sum_{1\leq i< j\leq n} \mathsf{x}_{ij}$ evaluated on $\mathds{1}_G$ is equal to the edge density,  $\frac{|E(G)|}{\binom{|V(G)|}{2}}$, of the graph
$G \in \mathcal{G}_n$. In order to show that the  edge density of any $G \in \mathcal{G}_n$ is at most
$\beta$ (here we would like $\beta$ to be $\frac{1}{2}+O(\frac{1}{n}))$, it suffices to find polynomials $\mathsf{r}_j$ such that
$$\beta-\mathsf{p}_n \equiv \sum_{j} \mathsf{r}_j^2 \ \mod \mathscr{I}_n^{K_3}.$$

Translating the flag sos proof in \cite{Falgas-RavryVaughan}, we obtain the following polynomial sos certificate

\begin{align} \label{equivalencemantel}
\frac12-\mathsf{p}_n + \mathsf{err}  \equiv\mathbb{E}_{\theta}\left[\begin{pmatrix}\mathsf{p}^{\theta}_{F_0} & \mathsf{p}^{\theta}_{F_1}\end{pmatrix}  \begin{pmatrix} \frac12   &-\frac12 \\ -\frac12 &\frac12\end{pmatrix} \begin{pmatrix}\mathsf{p}^{\theta}_{F_0} \\ \mathsf{p}^{\theta}_{F_1} \end{pmatrix}\right] +\frac13 \pp_{H_1}\textrm{ mod } \mathscr{I}_n^{K_3},
\end{align}
where $$ F_0 = \labeledvnonedge{}{1}, \,\,\,\, F_1 = \labeledvedge{}{1}, \,\,\,\,  H_1=\Hone,$$
and $\mathsf{err}(\mathds{1}_G)$ has value $O(\frac{1}{n})$ for every $G \in \mathcal{G}_n$. The first expression on the right hand side is a sos by construction. The second expression is
also a sum of squares because $\left(\mathsf{q}^\alpha_H\right)^2 \equiv \mathsf{q}^\alpha_H \mod \mathscr{I}_n^{K_3}$ since
$\mathsf{x}_{ij}$ and $(1-\mathsf{x}_{ij})$ are equivalent to their squares mod $\mathscr{I}_n^{K_3}$.
Therefore, 
\begin{align}
\mathsf{p}_{H}\equiv \frac{1}{a_{H} \binom{n}{|V(H)|}}\sum_{\alpha\in \textup{Inj}(V(H),[n])}\left(\mathsf{q}^\alpha_{H}\right)^2\    \mod \mathscr{I}_n^{K_3},
\end{align}
for every $H$. In particular, $\mathsf{p}_{H_1}$ is a sos mod $\mathscr{I}_n^{K_3}$. We have expressed the edge density expression on the left hand side as a polynomial sos modulo the ideal $\mathscr{I}_n^{K_3}$.
For a verification of this equivalence as well as a full translation of \cite{Falgas-RavryVaughan} to polynomials, see the Appendix.
\end{example}

\section{Sum of squares representations of invariant polynomials}
\label{sec:GatermannParrilo}

In \cite{GatermannParrilo}, Gatermann and Parrilo use methods from
representation theory to organize the computation of sos expressions for polynomials
that are invariant with respect to a finite group. These symmetry-reduction techniques allow the SDP that provides the potential sos to be broken into several smaller SDPs
that are coupled together, often leading to tremendous computational savings.
Since the graph density polynomials that arise in Tur{\'a}n problems are invariant
under an induced action of the symmetric group $\mathfrak{S}_n$ on $n$ letters, we can apply
the methods in \cite{GatermannParrilo} to provide
an alternate and systematic method for establishing graph density
inequalities.

Our main aim in this section is to describe the strategy and mechanics in
\cite{GatermannParrilo} specialized to our setting of $\mathfrak{S}_n$ acting on $\RR[{\bf x}]$. For a full proof of the Gatermann-Parrilo method specialized to our setting, see Appendix $A$ of \cite{RSST}. To keep the paper self-contained, we assume very little background.

\subsection{Representation theory of the symmetric group} \label{sec:representation theory basics}
There are many excellent expositions of the representation theory of $\mathfrak{S}_n$, and our brief account
below is based on \cite[Chapter 1]{SaganBook}. We will reference general theorems from \cite{SaganBook} even
if we only state their specialized versions for $\mathfrak{S}_n$.

Recall that $\RR[{\bf x}]$ denotes the polynomial ring in the variables $\mathsf{x}_{ij},  1 \leq i < j \leq n$. Since we will be searching for $d$-sos polynomials for a fixed degree $d$, we will focus on
$\RR[\mathbf{x}]_{\leq d}$, the set of all polynomials in $\RR[{\bf x}]$ of degree at most $d$. A natural basis for this vector space is the set of
all monomials of degree at most $d$, and hence the dimension of $V := \RR[\mathbf{x}]_{\leq d}$ is
$D := {e+d  \choose d}$ where $e := {n \choose 2}$. The symmetric group $\mathfrak{S}_n$ acts on monomials in
$\RR[{\bf x}]$ via $\mathfrak{s} \mathsf{x}_{ij} := \mathsf{x}_{\mathfrak{s}(i) \mathfrak{s}(j)}$ for each $\mathfrak{s} \in \mathfrak{S}_n$. Extending this action linearly to the vector space $V$ makes $V$ a $\mathfrak{S}_n$-module.
This means that the multiplication $\mathfrak{s} \mathsf{f}$ for
$\mathfrak{s} \in \mathfrak{S}_n$ and $\mathsf{f} \in V$ satisfies the following properties:
$$
(i) \,\mathfrak{s} \mathsf{f} \in V, \,\,\,\,
(ii) \,(\mathfrak{s} \mathfrak{t}) \mathsf{f} = \mathfrak{s} (\mathfrak{t} \mathsf{f}), \,\,\,\,
(iii) \,\mathfrak{e} \mathsf{f} = \mathsf{f}, \,\,\,\,
(iv) \, \mathfrak{s} ( \alpha \mathsf{f} + \beta \mathsf{g} ) = \alpha \mathfrak{s} \mathsf{f} + \beta \mathfrak{s}  \mathsf{g}
$$
for all $\mathfrak{s}, \mathfrak{t} \in \mathfrak{S}_n, \mathsf{f}, \mathsf{g} \in V, \alpha, \beta \in \RR$,
and where $\mathfrak{e}$ is the identity permutation in $\mathfrak{S}_n$. The $\mathfrak{S}_n$-module $V$ is called
the {\em permutation representation} of $\mathfrak{S}_n$ associated to the monomials of degree at most $d$ for reasons
we will see below.

The $\mathfrak{S}_n$-module $V$ gives rise to a homomorphism $\vartheta \,:\, \mathfrak{S}_n \rightarrow \textup{GL}(V)$,
where $\textup{GL}(V)$ is the set of invertible linear transformations from $V$ to itself,
by defining $\vartheta(\mathfrak{s})$ to be the linear transformation of $V$ corresponding to multiplication by $\mathfrak{s}$. The matrices realizing $\vartheta(\mathfrak{s})$ for all $\mathfrak{s} \in \mathfrak{S}_n$, with respect to a fixed
basis of $V$, form a set of {\em representing matrices} of $\vartheta$.
For example, the representing matrices of $\vartheta$, with respect to the monomial basis of $V$,  are the permutation matrices of size $D \times D$.
This follows since for each $\mathfrak{s} \in \mathfrak{S}_n$, $\vartheta(\mathfrak{s})$
sends a monomial to a monomial. Let $P_{\mathfrak{s}} \in \RR^{D \times D}$ denote the permutation matrix
representing $\vartheta(\mathfrak{s})$.  Note that the matrices $P_{\mathfrak{s}}$ are orthonormal.

Conversely, a homomorphism $\vartheta \,:\, \mathfrak{S}_n \rightarrow \textup{GL}(V)$ makes $V$
a $\mathfrak{S}_n$-module via the multiplication $\mathfrak{s}\mathsf{f} := \vartheta(\mathfrak{s}) \mathsf{f}$ for each
$\mathfrak{s} \in \mathfrak{S}_n$ and $\mathsf{f} \in V$.
By this discussion, a representation of $\mathfrak{S}_n$ refers to the $\mathfrak{S}_n$-module $V$, or the homomorphism
$\vartheta \,:\, \mathfrak{S}_n \rightarrow \textup{GL}(V)$, or even a set of representing matrices of $\vartheta$
with respect to a fixed basis of $V$.
The {\em trivial representation} of $\mathfrak{S}_n$ is the homomorphism
$\vartheta(\mathfrak{s}) = 1$ for all $\mathfrak{s} \in \mathfrak{S}_n$. Equivalently, a $\mathfrak{S}_n$-module $V$ is a trivial representation of $\mathfrak{S}_n$ if $V$ is one-dimensional and $\mathfrak{s} \mathsf{f} = \mathsf{f} $ for all $\mathfrak{s} \in \mathfrak{S}$ and $\mathsf{f} \in V$.

A subspace $W$ of $V$ is a $\mathfrak{S}_n$-{\em submodule} if it is invariant under the action of $\mathfrak{S}_n$, i.e., $\mathfrak{s} \mathsf{f} \in W$ for all $\mathsf{f} \in W$ and $\mathfrak{s} \in \mathfrak{S}_n$.
A $\mathfrak{S}_n$-module is {\em irreducible} if it does not contain any nontrivial submodules, and the associated
homomorphism $\vartheta$ is also said to be irreducible. The irreducible $\mathfrak{S}_n$-modules are indexed by the partitions $\bm{\lambda} = (\lambda_1, \lambda_2, \ldots, \lambda_k)$ of $n$ denoted as $\bm{\lambda} \vdash n$.
There is a canonical irreducible $\mathfrak{S}_n$-module indexed by the partition $\bm{\lambda}$ called the
{\em Specht module} $S^{\bm{\lambda}}$ whose dimension is $n_{\bm{\lambda}}$, the number of {\em standard tableaux} of shape $\bm{\lambda}$. All irreducible representations of $\mathfrak{S}_n$ are isomorphic to one of these Specht modules. See \cite[Chapter 2]{SaganBook} for a detailed account of the combinatorics underlying the representation theory of $\mathfrak{S}_n$.  In Section~\ref{sec:main results}, we give more details of the specific items we will need.

One of the fundamental results in the representation theory of finite groups specialized to our setting says the following.

\begin{theorem} [Maschke's theorem]  \cite[Theorem~1.5.3]{SaganBook} The $\mathfrak{S}_n$-module $V = \RR[\mathbf{x}]_{\leq d}$ breaks into a direct sum of irreducible submodules.
\end{theorem}

Let $V_{\bm{\lambda}} := \oplus_{i=1}^{m_{\bm{\lambda}}} V_{\bm{\lambda}}^i$
denote the sub-sum of all the isomorphic copies of the irreducible $S^{\bm{\lambda}}$ in a
full irreducible decomposition of $V$ from Maschke's theorem.
While this decomposition is not unique, the {\em multiplicity} $m_{\bm{\lambda}}$ of $S^{\bm{\lambda}}$ in
the decomposition is. The subspace $V_{\bm{\lambda}}$ is called an isotypic of $V$, and the
decomposition
\begin{align} \label{eq:isotypic decomposition}
V = \oplus_{\bm{\lambda} \vdash n}  V_{\bm{\lambda}}
\end{align}
called the {\em isotypic decomposition} of $V$, is unique. A useful fact to note is that a $\mathfrak{S}_n$-invariant polynomial $\mathsf{f}$ (i.e., $\mathfrak{s} \mathsf{f}=\mathsf{f}$ for each $\mathfrak{s}\in \mathfrak{S}_n$) must lie in the isotypic corresponding to the trivial representation of $\mathfrak{S}_n$.

The irreducible decomposition of $V$ guaranteed by Maschke's theorem creates several block-diagonal structures that
are the key to the methods in \cite{GatermannParrilo}. Let
\begin{align} \label{eq:irreducible decomposition}
V = \oplus_{\bm{\lambda} \vdash n} \oplus_{i=1}^{m_{\bm{\lambda}}} V_{\bm{\lambda}}^i
\end{align}
be the  full  decomposition of $V$ into irreducibles
where $V_{\bm{\lambda}}^i \cong S^{\bm{\lambda}}$ for $i=1,\ldots, m_{\bm{\lambda}}$.

The first instance of block structure arises at the level of representing matrices of the homomorphism $\vartheta \,:\, \mathfrak{S}_n \rightarrow \textup{GL}(V)$. Suppose $\mathcal{B}$ is a basis of $V$ obtained by concatenating bases of the different irreducible submodules
in \eqref{eq:irreducible decomposition}. Then the representing matrices of $\vartheta(\mathfrak{s})$  with respect to $\mathcal{B}$ are block-diagonal with a block corresponding to each irreducible $V_{\bm{\lambda}}^i$ in \eqref{eq:irreducible decomposition} of size $n_{\bm{\lambda}} \times n_{\bm{\lambda}}$.
The basis $\mathcal{B}$ is said to be {\em symmetry-adapted} if the following stronger property holds:
for a fixed $\bm{\lambda}$, the $n_{\bm{\lambda}} \times n_{\bm{\lambda}}$ sized blocks corresponding to  the $ m_{\bm{\lambda}}$ irreducibles $V_{\bm{\lambda}}^i$ are {\em exactly the same}, i.e., the representing matrix of $\vartheta(\mathfrak{s})$
with respect to $\mathcal{B}$ has the form:
\begin{align} \label{eq:SA representing matrices}
B_{\mathfrak{s}} = \left[ \begin{array}{ccc|ccc|c}
B_{\bm{\lambda_1}} & \cdots & 0 & 0 & 0 & 0 & \cdots \\
0 & B_{\bm{\lambda_1}} & \cdots & 0 & 0 & 0 & \cdots \\
0  & 0  & \ddots & 0 & 0 & 0 & \cdots \\
\hline
0 & 0 & 0 & B_{\bm{\lambda_2}} & \cdots & 0 & \cdots \\
0 & 0 & 0 & 0 & B_{\bm{\lambda_2}}& \cdots& \cdots  \\
0 & 0 & 0 & 0  & 0 & \ddots & \cdots \\
\hline
\vdots & \vdots &\vdots &\vdots &\vdots &\vdots & \ddots
 \end{array} \right].
\end{align}
It is usual to notate this as
\begin{align}  \label{eq:SA blocks in Bs}
B_{\mathfrak{s}} = \oplus_{\bm{\lambda} \vdash n} \oplus_{i=1}^{m_{\bm{\lambda}}} B_{\bm{\lambda}}, \,\,\,\,\,\textup{ where } B_{\bm{\lambda}} \in \RR^{n_{\bm{\lambda}} \times n_{\bm{\lambda}}}.
\end{align}
A symmetry-adapted basis of $V$ always exists and an algorithm to find it is given in
\cite[Chapter 5.2]{FaesslerStiefelBook}. Recall that the permutation matrices $P_{\mathfrak{s}}$ were representing matrices for $\vartheta$, but they are not block-diagonal.
If $M$ is the change of basis matrix from the monomial basis of $V$ to $\mathcal{B}$, then the new representing matrices are $B_{\mathfrak{s}} := M P_{\mathfrak{s}} M^{-1}$. If $M$ is orthogonal, then $B_{\mathfrak{s}} = M P_{\mathfrak{s}} M^\top$ is also
orthogonal.

Next we consider the set of matrices that commute with every $B_{\mathfrak{s}}$. This is the {\em commutant algebra}
\begin{align} \label{eq:commutant algebra}
\textup{Com } B := \{ Q \in \RR^{D \times D} \,:\, Q B_{\mathfrak{s}} = B_{\mathfrak{s}} Q \,\,\forall \,\,\mathfrak{s} \in \mathfrak{S}_n \}.
\end{align}
Matrices in $\textup{Com } B$ have a very special structure as a consequence of the block-diagonal nature of $B_{\mathfrak{s}}$ and Schur's Lemma \cite[Theorem 1.6.5]{SaganBook}.

\begin{theorem} \cite[Thm~1.7.8 (2)]{SaganBook}\label{thm:block structure of commutants}
If $Q$ lies in the commutant algebra $\textup{Com } B$, then $Q$ is block-diagonal with a block
$Q_{\bm{\lambda}}$ for each partition $\bm{\lambda} \vdash n$.  Further,
$Q_{\bm{\lambda}}$ is a block matrix with $m_{\bm{\lambda}}$ row and column blocks each of size
$n_{\bm{\lambda}} \times n_{\bm{\lambda}}$. The matrices in each block of $Q_{\bm{\lambda}}$ are multiples of the
identity matrix $I_{n_{\bm{\lambda}}}$.
\end{theorem}

For example, if $m_{\bm{\lambda}} = 3$ and
$n_{\bm{\lambda}} = 2$, then $Q_{\bm{\lambda}}$ is the $3 \times 3$ block matrix shown below on the left.
Now notice that  by permuting rows and columns, we can transform $Q_{\bm{\lambda}}$ to
a block-diagonal matrix with $n_{\bm{\lambda}}$ {\em equal} blocks each of size $m_{\bm{\lambda}} \times m_{\bm{\lambda}}$. This is the block-diagonalization of $Q$ needed in \cite{GatermannParrilo}.
$$
\left[ \begin{array}{cc|cc|cc}
c_1 & 0    & c_2 & 0     & c_3 & 0    \\
0    & c_1 &  0    & c_2 &  0    & c_3 \\
\hline
c_4 & 0    & c_5 & 0     & c_6 & 0    \\
0    & c_4 &  0    & c_5 &  0    & c_6 \\
\hline
c_7 & 0    & c_8 & 0     & c_9 & 0    \\
0    & c_7 &  0    & c_8 &  0    & c_9
\end{array} \right]
\,\,\longrightarrow \,\,
\left[ \begin{array}{ccc|ccc}
c_1 & c_2    & c_3 & 0     & 0 & 0    \\
c_4    & c_5 &  c_6    & 0 &  0    & 0 \\
c_7 & c_8   & c_9 & 0     & 0 & 0    \\
\hline
0    & 0 &  0    & c_1 &  c_2    & c_3 \\
0 & 0    & 0 & c_4     & c_5 & c_6   \\
0    & 0 &  0    & c_7 &  c_8    & c_9
\end{array} \right]  .
$$

The above discussion can be phrased more generally as follows.

\begin{theorem}[The Fundamental Theorem] \cite[pp 40]{FaesslerStiefelBook} \label{thm:Fundamental Theorem}
Consider the decomposition of $V = \RR[\mathbf{x}]_{\leq d}$ as in \eqref{eq:irreducible decomposition}
under the representation $\vartheta \,:\, \mathfrak{S}_n \rightarrow \textup{GL}(V)$ with representing
matrices $B_\mathfrak{s}$ for each $\mathfrak{s} \in \mathfrak{S}_n$ computed with respect to a symmetry-adapted basis of $V$.  Suppose $Q \in \RR^{D \times D}$ is such that
$Q B_{\mathfrak{s}} = B_{\mathfrak{s}} Q$ for all $\mathfrak{s} \in \mathfrak{S}_n$.
Then there is a reordering of the symmetry-adapted basis with respect to which
$Q$ is block-diagonal of the form:
\begin{align} \label{eq:GP block structure matrix form}
Q = \oplus_{\bm{\lambda} \vdash n} \oplus_{i=1}^{n_{\bm{\lambda}}} Q_{\bm{\lambda}}.
\end{align}
\end{theorem}

Note that the structure of $Q$ in \eqref{eq:GP block structure matrix form} is doubly block-diagonal similar to \eqref{eq:SA representing matrices}.
The matrix $Q_{\bm{\lambda}}$ has size $m_{\bm{\lambda}} \times m_{\bm{\lambda}}$. Also, since the algorithm creates a bijection between the matrices $Q_{\bm{\lambda}}$  in \eqref{eq:GP block structure matrix form} and the $n_{\bm{\lambda}}$ standard tableaux $\tau_{\bm{\lambda}}^1, \ldots, \tau_{\bm{\lambda}}^{n_{\bm{\lambda}}}$, we may rewrite
\eqref{eq:GP block structure matrix form} as
\begin{align} \label{eq:GP block structure matrix form v2}
Q = \oplus_{\bm{\lambda} \vdash n} \oplus_{\tau_{\bm{\lambda}}} Q_{\bm{\lambda}}.
\end{align}

For the particular ordering of the symmetry-adapted basis required to block-diagonalize $Q$, we refer the reader to
\cite[pp 40]{FaesslerStiefelBook}. As a consequence of the reordering we get a different direct sum decomposition of each isotypic $V_{\bm{\lambda}}$ in
\eqref{eq:irreducible decomposition}, indexed by the standard tableaux  of shape $\bm{\lambda}$:
\begin{align}\label{eq:alternate-decomposition}
V_{\bm{\lambda}} = \oplus_{i=1}^{n_{\bm{\lambda}}} W_{\tau_{\bm{\lambda}}^i} \,\,\,\textup{ leading to }
V = \oplus_{\bm{\lambda} \vdash n} \oplus_{i=1}^{n_{\bm{\lambda}}} W_{\tau_{\bm{\lambda}}^i}.
\end{align}

Note that this vector space decomposition is not a $\mathfrak{S}_n$-module decomposition. For more information on the vector spaces $W_{\tau_{\bm{\lambda}}}$, see Section 2 and Appendix A of \cite{RSST}.

\subsection{Sum of squares for an invariant polynomial} \label{sec:GP method}
 The main message of \cite{GatermannParrilo} is that the computation of sos certificates of degree at most $d$ for an invariant polynomial can be helped greatly by  the irreducible decomposition \eqref{eq:irreducible decomposition} and the block structures discussed above. We present their strategy in our setting.

A representation $\vartheta$ of $\mathfrak{S}_n$ on $\RR[\mathbf{x}]_{\leq d}$ induces a representation of
$\mathfrak{S}_n$ on $\mathcal{S}^D$, the vector space of $D \times D$ real symmetric matrices as follows:
\begin{align} \label{eq:action on matrices}
 \mathfrak{s}  X := \vartheta(\mathfrak{s})^\top X \vartheta(\mathfrak{s}) \,\,\,\,\forall \,\,\, \mathfrak{s} \in \mathfrak{S}_n, X \in \mathcal{S}^D.
 \end{align}
Here we identity $\vartheta(\mathfrak{s})$ with a matrix representation of it, and assume that these matrices are all orthonormal. This action preserves the cone of psd
matrices. The set of all invariant matrices under the action \eqref{eq:action on matrices} are those of the form
$X = \vartheta(\mathfrak{s})^\top X \vartheta(\mathfrak{s})$ for all $\mathfrak{s} \in \mathfrak{S}_n$. Note that these are precisely the matrices that commute with every $\vartheta(\mathfrak{s})$ and hence they can be block-diagonalized as in  \eqref{eq:GP block structure matrix form v2}.  An easy way to construct an invariant matrix from any matrix $X \in \mathcal{S}^D$ is to pass to its {\em symmetrization}
$\bar{X} := \sum_{\mathfrak{s} \in \mathfrak{S}_n} \vartheta(\mathfrak{s})^\top X \vartheta(\mathfrak{s})$. If $X$ is psd then so is $\bar{X}$.

Suppose we are given a $\mathfrak{S}_n$-invariant polynomial $\mathsf{f}(\mathbf{x}) \in \RR[{\bf x}]$ for which we wish to find a $d$-sos representation modulo an ideal $\mathscr{I} \subset \RR[\mathbf{x}]$ that is also
$\mathfrak{S}_n$-invariant. Let $[\mathbf{x}]_{\leq d}$ denote the vector of monomials in $\RR[\mathbf{x}]_{\leq d}$. The polynomial $\mathsf{f}(\mathbf{x})$ is $d$-sos modulo  $\mathscr{I}$ if and only if there exists a psd matrix $Q$ such that
\begin{align*} \label{eq:sos form}
\mathsf{f}(\mathbf{x}) \equiv [\mathbf{x}]_{\leq d}^\top Q [\mathbf{x}]_{\leq d} \,\,\,\textup{ mod } \mathscr{I} \,\,\,:\Leftrightarrow \,\,\, \mathsf{f}(\mathbf{x}) - \mathsf{h} = [\mathbf{x}]_{\leq d}^\top Q [\mathbf{x}]_{\leq d} \textup{ for some } \mathsf{h} \in \mathscr{I}.
\end{align*}
Since $\mathsf{f}$ and $\mathscr{I}$ are both $\mathfrak{S}_n$-invariant, we may symmetrize both sides of the above equation to get another sos expression for $\mathsf{f}$ mod $\mathscr{I}$. Therefore, we can assume that the sos on the right is $\mathfrak{S}_n$-invariant.
In other words, we may assume that the psd matrix $Q$ in
$[\mathbf{x}]_{\leq d}^\top Q [\mathbf{x}]_{\leq d}$ is symmetrized since symmetrizing the sos yields
\begin{align*}
\sum_{\mathfrak{s} \in \mathfrak{S}_n} (\vartheta(\mathfrak{s}) [\mathbf{x}]_{\leq d})^\top Q (\vartheta(\mathfrak{s})[\mathbf{x}]_{\leq d}) =
\sum_{\mathfrak{s} \in \mathfrak{S}_n} ([\mathbf{x}]_{\leq d})^\top \left( \vartheta(\mathfrak{s})^\top Q \vartheta(\mathfrak{s}) \right)[\mathbf{x}]_{\leq d}\\
 = ([\mathbf{x}]_{\leq d})^\top \left( \sum_{\mathfrak{s} \in \mathfrak{S}_n}  \vartheta(\mathfrak{s})^\top Q \vartheta(\mathfrak{s}) \right) [\mathbf{x}]_{\leq d}.
\end{align*}

Now suppose $M$ denotes the change of basis matrix from the monomial basis of $\RR[\mathbf{x}]_{\leq d}$ to a symmetry-adapted basis with respect to \eqref{eq:irreducible decomposition}.
Then $\mathsf{f}(\mathbf{x})$ is sos mod $\mathscr{I}$ if and only if there is an invariant psd matrix $Q$ such that
\begin{align*}
\mathsf{f}(\mathbf{x}) & \equiv  [\mathbf{x}]_{\leq d}^\top Q [\mathbf{x}]_{\leq d} \,\,\,\, \textup{ mod } \mathscr{I}\\
      & \equiv  [\mathbf{x}]_{\leq d}^\top M M^\top Q M M^\top [\mathbf{x}]_{\leq d} \,\,\,\, \textup{ mod } \mathscr{I}\\
      & \equiv  (M^\top[\mathbf{x}]_{\leq d})^\top (M^\top Q M) (M^\top [\mathbf{x}]_{\leq d}) \,\,\,\, \textup{ mod } \mathscr{I}\\
      & \equiv  \mathbf{y}^\top \tilde Q \mathbf{y} \,\,\,\, \textup{ mod } \mathscr{I}
\end{align*}
where $\mathbf{y} := M^\top [\mathbf{x}]_{\leq d}$ is the vector of elements in the symmetry-adapted basis of $V$. The matrix $\tilde{Q}=M^\top Q M$ has the block structure in \eqref{eq:GP block structure matrix form v2} since $Q$ is invariant and $M^{\top} Q M$ is precisely the transformation that block diagonalizes $Q$.
The components of $\mathbf{y}$ will be referred to as {\em (symmetry-adapted) basis polynomials}. The block structure of $\tilde{Q}$ endows a block structure on $\mathbf{y}$, with big blocks
indexed by the partitions $\bm{\lambda}$ of $n$, and each block $\mathbf{y}_{\bm{\lambda}}$ broken further into blocks indexed by the standard tableaux of shape $\bm{\lambda}$.

For simplicity, we rename $\tilde{Q}$ by $Q$ and rewrite the sos expression for $\mathsf{f}$
using \eqref{eq:GP block structure matrix form v2}. Thus

\begin{align}
\mathsf{f}(\mathbf{x})
&\equiv \mathbf{y}^\top Q {\mathbf{y}}
\equiv \sum_{\bm{\lambda}} \sum_{\tau_{\bm{\lambda}}} {\mathbf{y}}_{\tau_{\bm{\lambda}}}^\top  {Q}_{\bm{\lambda}} {\mathbf{y}}_{\tau_{\bm{\lambda}}}  \\
&\equiv \sum_{\bm{\lambda}} \sum_{\tau_{\bm{\lambda}}} \langle {Q}_{\bm{\lambda}},  {\mathbf{y}}_{\tau_{\bm{\lambda}}}  {\mathbf{y}}_{\tau_{\bm{\lambda}}}^\top  \rangle \equiv \sum_{\bm{\lambda}} n_{\bm{\lambda}} \langle {Q}_{\bm{\lambda}},   {Y}_{\bm{\lambda}}  \rangle \,\,\,\,\textup{ mod } \mathscr{I} \label{onlyonetableau}
\end{align}
where ${Q}_{\bm{\lambda}}$ is an unknown psd matrix of
 size $m_{\bm{\lambda}} \times m_{\bm{\lambda}}$, and
 $${Y}_{\bm{\lambda}} := \frac{1}{|\mathfrak{S}_n|}\sum_{\mathfrak{s}\in \mathfrak{S}_n} \mathfrak{s}({\mathbf{y}}_{\tau_{\bm{\lambda}}'} {\mathbf{y}}_{\tau_{\bm{\lambda}}'}^\top)$$ is a matrix of the same size where $\tau_{\bm{\lambda}}'$ is any tableau of shape $\bm{\lambda}$ (symmetrization here is possible since $\mathsf{f}$ is $\mathfrak{S}_n$-invariant).  The choice of tableaux $\tau_{\bm{\lambda}}'$ does not affect $Y_{\bm{\lambda}}$. For further explanations, see Appendix A of \cite{RSST}.
 Therefore, in order to check if $\mathsf{f}$ is a sos of the above type, we need to search for a psd matrix ${Q}_{\bm{\lambda}}$ of size $m_{\bm{\lambda}}$ for each $\bm{\lambda}$ such that the linear equations that come from equating
 $\mathsf{f}$ to the sos expression \eqref{onlyonetableau} hold. For details on how to set up this SDP, we refer the reader to
 \cite[Chapter 7.2.1]{BPTSIAMBook}.

It might be possible to certify the nonnegativity of $\mathsf{f}({\bf x})$ by using only a subset
$\Lambda$ of the partitions of $n$.  Since we will rely on such subsets in the next two sections, we make a formal definition to say precisely what we mean.

\begin{definition}\label{def:restrictedGP}
We say that a sos expression for $\mathsf{f}$ modulo the ideal $\mathscr{I}$ can be obtained through the Gatermann-Parrilo SDP {\em restricted to partitions in} $\Lambda$ if there exists psd matrices $Q_{\bm{\lambda}}$ such that
\begin{equation}\label{eq:restrictedGP}
\mathsf{f}(\mathbf{x}) \equiv \sum_{\bm{\lambda} \in \Lambda} \langle {Q}_{\bm{\lambda}},   Y_{\bm{\lambda}}
\rangle \,\,\,\,\textup{ mod } \mathscr{I}.
\end{equation}

\end{definition}

\section{Main Results} \label{sec:main results}

In this section, we establish the connection between the sos certificate \eqref{razborovsosshape} obtained from the flag algebra method to those that can be obtained from the Gatermann-Parrilo symmetry-adapted SDP described in the previous section. In particular, we show that this sos certificate can be obtained from a SDP restricted to a fixed number of \textup{known} partitions as in Definition \ref{def:restrictedGP} above.

Note that the sos certificate \eqref{razborovsosshape} is a nested sos (sums of sums of squares really)
which makes it cumbersome to work with. Therefore, we work with the inner sums of squares in \eqref{razborovsosshape} whenever possible. In Theorem~\ref{lcorbits} and Corollary \ref{cor:orbitpolyvectorspace}, we use the  innermost sos by fixing some $\theta, \sigma$ and $l$, namely

\begin{equation}\label{deepsos}
{\mathbf{p}^{\theta, \sigma, l}}^\top Q_{\sigma, l} \mathbf{p}^{\theta, \sigma, l}=:\sum_{j} \mathsf{r}_j^2,
\end{equation}
to prove that each $\mathsf{r}_j$ is invariant under a particular subgroup of $\mathfrak{S}_n$, and as such, lies in the direct sum of finitely many $V_{\bm{\mu}}$ which we explicitly describe.

In Theorem~\ref{cor:sos from same isotypics}, we need the sos \eqref{razborovsosshape} obtained from fixing only $\sigma$ and $l$ and which, with a slight abuse of notation, we denote by
\begin{equation}\label{middlesos}
\sum_{\theta}\sum_{j} \mathsf{r}_{\theta,j}^2 := \mathbb{E}_{\theta}{\mathbf{p}^{\theta, \sigma, l}}^\top Q_{\sigma,l} \mathbf{p}^{\theta, \sigma, l}.
\end{equation}

We prove that such a sos can be retrieved through the Gatermann-Parrilo method restricted to partitions said to be lexicographically greater or equal to $(n-k,1^k)$ where $k$ is the size of $\sigma$ (see definition below).

Corollary \ref{wholething} involves the conic combination of \eqref{razborovsosshape} for problems where several types $\sigma$ and flags of different size are necessary. It shows that this whole flag sos can also be retrieved through the Gatermann-Parrilo method restricted to partitions lexicographically greater or equal to $(n-k^*,1^{k^*})$ where $k^*$ is the maximum size among all types present. A key feature is that the number of such partitions is independent of $n$.

Throughout this section, we assume that $V=\RR[\mathbf{x}]_{\leq d}$ is such that $d$ is at least as big as the maximum number of edges in the flags considered in the conic combination of sos expressions of the form \eqref{razborovsosshape}.

\subsection{Invariance of density polynomials} A \emph{partition} of $n$ is a way of writing $n$ as a sum of positive integers; each summand is called a \emph{part}. We denote a partition $\bm{\lambda}$ by a vector containing the parts $\lambda_i$ in non-decreasing order, i.e., $(\lambda_1, \ldots, \lambda_t)$ such that $\lambda_1\geq \ldots\geq \lambda_t >0$ and $\lambda_1+\ldots+\lambda_t=n$. A key partition for us is the \emph{hook partition} for which all the parts but one are $1$; we denote the hook partition with $k+1$ parts as $(n-k,1^k)$. There is a lexicographic order on the partitions of $n$; for $\bm{\lambda}=(\lambda_1, \ldots, \lambda_{t_1})$ and $\bm{\mu}=(\mu_1,\ldots, \mu_{t_2})$, we write $\bm{\lambda}\geq_{\textup{lex}} \bm{\mu}$ if the vector $(\lambda_1, \ldots, \lambda_{t_1})$ is lexicographically greater than or equal to the vector $(\mu_1, \ldots, \mu_{t_2})$. A partition $\bm{\lambda}$ has a shape (Young diagram) with rows of size $\lambda_1 \geq \lambda_2 \geq \cdots \geq \lambda_t$. A {\em tableau} of shape $\bm{\lambda}$, denoted as $\tau_{\bm{\lambda}}$,  is a filling of the boxes in the diagram of ${\bf \lambda}$ by the numbers $1, \ldots,n$. The tableau is {\em standard} if the numbering increases along each row and column. The {\em row group} of a tableau $\tau_{\bm{\lambda}}$ is the subgroup of  $\mathfrak{S}_n$
defined as
$$\mathfrak{R}_{\tau_{\bm{\lambda}}} := \{ \mathfrak{s} \in \mathfrak{S}_n: \mathfrak{s} \textup{ fixes the set of numbers in each row of } \tau_{\bm{\lambda}}\}.$$
Note that this group is isomorphic to $\mathfrak{S}_{\bm{\lambda}} := \mathfrak{S}_{\lambda_1} \times \cdots \times \mathfrak{S}_{\lambda_t}$.

We will show that the polynomials $\mathsf{p}_F^\theta$ from \eqref{dthetaf} are invariant under a particular row group for all flags $F\in \mathcal{F}_l^\sigma$.  Recall that these polynomials were defined
from the choice of a $\sigma$-flag $F$ of size $l$, and an injective map $\theta\in\textup{Inj}([k],[n])$. In particular, no partitions, tableaux or row groups were involved. We first present an example.

\begin{example} \label{exorbit_part2}
Consider $n=5$ and the hook partition $\bm{\lambda}=(2,1,1,1)$. The row group of the tableau
$$\tau_{\bm{\lambda}}=\young(45,1,2,3) \,\,\,\,\textup{ is }\,\,\,\,
\mathfrak{R}_{\tau_{\bm{\lambda}}}=\{1, (4,5)\}.$$

Suppose we choose the type $\sigma = \labeledcherries{1}{2}{3}$, set $l=4$, and  consider the $\sigma$-flag
$F =$ \labeledclaw{$1$}{$2$}{$3$}{}. Since $n=5$, assume $\theta \,:\, [3] \rightarrow [5]$
is such that $\theta(1)=1$, $\theta(2)=2$ and $\theta(3)=3$.
The set $\textup{Inj}_\theta(V(F),[5])$ contains two maps, both preserving $\theta$ and hence, sending $1,2,3$ to themselves. Suppose $\alpha_1 \in \textup{Inj}_\theta(V(F),[5])$ sends
the unlabeled vertex in $F$ to $4$ and $\alpha_2 \in \textup{Inj}_\theta(V(F),[5])$ sends it to $5$.

Using the formula in \eqref{eq:qhf}, one can see that
\begin{align*}
&\mathsf{q}^{\alpha_1}_F = \mathsf{x}_{12}\mathsf{x}_{13}\mathsf{x}_{14}(1-\mathsf{x}_{23})(1-\mathsf{x}_{24})(1-\mathsf{x}_{34}), \textup{ and }\\
&\mathsf{q}^{\alpha_2}_F = \mathsf{x}_{12}\mathsf{x}_{13}\mathsf{x}_{15}(1-\mathsf{x}_{23})(1-\mathsf{x}_{25})(1-\mathsf{x}_{35}).
\end{align*}
Similarly, from \eqref{dthetaf} we get that
$$\mathsf{p}_F^\theta = \frac{1}{|\textup{Inj}_\theta(V(F),[n])|}\sum_{\alpha \in\textup{Inj}_\theta(V(F),[n])} \mathsf{q}^\alpha_F= \frac{1}{2} (\mathsf{q}^{\alpha_1}_F  + \mathsf{q}^{\alpha_2}_F).$$
Note that $\mathsf{p}_F^\theta$ is invariant under $\mathfrak{R}_{\tau_{\bm{\lambda}}}$. Indeed, the action $1\in \mathfrak{R}_{\tau_{\bm{\lambda}}}$ applied to $\mathsf{p}_F^\theta$ doesn't change anything, and the action $(4,5)\in \mathfrak{R}_{\tau_{\bm{\lambda}}}$ sends $\mathsf{q}^{\alpha_1}_F$ to $\mathsf{q}^{\alpha_2}_F$ and vice-versa, thus leaving $\mathsf{p}_F^\theta$ unchanged.
\end{example}

We now prove that this observation is not accidental;
a polynomial $\pp^{\theta}_F$ is $\mathfrak{R}_{\tau_{\bm{\lambda}}}$-invariant for some tableau $\tau_{\bm{\lambda}}$ of shape ${\bm \lambda}$ where $\bm{\lambda}$ is a specific hook partition.

\begin{proposition}\label{pvsorbit}
Let $F$ be a $\sigma$-flag where $|\sigma|=k$, and $\theta:[k]\rightarrow [n]$ be an injective map. Consider the hook partition $\bm{\lambda}=(n-k,1^k)$ and a tableau $\tau_{\bm{\lambda}}$ where we fill the first row by numbers from $[n]\setminus\{\theta(i):i\in [k]\}$ in any order, and the $k$ remaining rows of size one with numbers from $\{\theta(i):i\in [k]\}$ in any order.  Then the polynomial $\pp^{\theta}_F$ is $\mathfrak{R}_{\tau_{\bm{\lambda}}}$-invariant.
\end{proposition}

\begin{proof}
Consider any $\mathfrak{s}\in \mathfrak{R}_{\tau_{\bm{\lambda}}} \cong \mathfrak{S}_{n-k}$ and $\alpha:V(F)\rightarrow [n]$ in $\textup{Inj}_\theta(V(F),[n])$. Then $(\mathfrak{s}\circ \alpha):V(F)\rightarrow [n]$ is also an injective map preserving $\theta$. Indeed, for any $\alpha\in \textup{Inj}_\theta(V(F),[n])$, the composition gives a map from $\mathfrak{R}_{\tau_{\bm{\lambda}}}\rightarrow \textup{Inj}_\theta(V(F),[n])$ which is $\frac{|\mathfrak{R}_{\tau_{\bm{\lambda}}}|}{|\textup{Inj}_\theta(V(F),[n])|}$-to-one surjective. This is because there are $n-l$ elements of $[n]$ outside the range of $\alpha$ and
$\alpha = \mathfrak{s} \circ \alpha$ for any
$\mathfrak{s}\in \mathfrak{R}_{\tau_{\bm{\lambda}}}$ that fixes the range of $\alpha$. Since there are $$(n-l)! = \frac{(n-k)!}{(n-k)(n-k-1) \cdots (n-l+1)} = \frac{|\mathfrak{R}_{\tau_{\bm{\lambda}}}|}{|\textup{Inj}_\theta(V(F),[n])|}$$
such permutations $\mathfrak{s}$, we get that the $\mathfrak{R}_{\tau_{\bm{\lambda}}}$-invariant polynomial
\begin{align*}
\frac{1}{|\mathfrak{R}_{\tau_{\bm{\lambda}}}|} \sum_{\mathfrak{s}\in {\mathfrak{R}_{\tau_{\bm{\lambda}}}}} \mathfrak{s}\cdot\mathsf{q}_F^\alpha=\frac{1}{|\mathfrak{R}_{\tau_{\bm{\lambda}}}|} \cdot \frac{|\mathfrak{R}_{\tau_{\bm{\lambda}}}|}{|\textup{Inj}_\theta(V(F),[n])|}   \sum_{\alpha\in \textup{Inj}_\theta(V(F),[n])}\mathsf{q}_F^\alpha
=\pp_F^{\theta}
\end{align*}
by \eqref{dthetaf},
which implies that $\mathsf{p}_F^\theta$ is $\mathfrak{R}_{\tau_{\bm{\lambda}}}$-invariant.
\end{proof}

\begin{theorem}\label{lcorbits}
Each polynomial $\mathsf{r}_j$ in the sos \eqref{deepsos}
is invariant under the row group $\mathfrak{R}_{\tau_{\bm{\lambda}}}$ corresponding to the tableau $\tau_{\bm{\lambda}}$ and hook partition $\bm{\lambda}$ as in Proposition~\ref{pvsorbit}.
\end{theorem}

\proof Since each $\mathsf{r}_j$ is a linear combination of $\mathsf{p}_F^\theta$ with $F\in \mathcal{F}_{l}^\sigma$, and since $\tau_{\bm{\lambda}}$ only depends on $\theta$, all $\mathsf{p}_F^\theta$ for any $F\in \mathcal{F}_{l}^\sigma$ are invariant under the same $\mathfrak{R}_{\tau_{\bm{\lambda}}}$. Therefore, $\mathsf{r}_j$ is also $\mathfrak{R}_{\tau_{\bm{\lambda}}}$-invariant.
\qed

\begin{remark} Note that the hook $\bm{\lambda}$ in Proposition~\ref{pvsorbit} and Theorem~\ref{lcorbits}
only depends on $k$, the size of the type $\sigma$ of the flags $F$ and not on their size $l$.
\end{remark}

\subsection{$\mathfrak{R}_{\tau_{\bm{\lambda}}}$-invariant polynomials and the isotypic decomposition}\label{sec:razborov-gatermann-parrilo}

Our next goal is to show that a $\mathfrak{R}_{\tau_{\bm{\lambda}}}$-invariant polynomial $\mathsf{f}$ lies in the span of certain specific isotypics in the isotypic decomposition \eqref{eq:isotypic decomposition} of
$V = \mathbb{R}[\mathbf{x}]_{\leq d}$. We will use this in the next subsection to prove that Razborov's sos can be
obtained by restricting the Gatermann-Parrilo SDP to the subblocks indexed by these isotypics/partitions.

Recall that the induced $\mathfrak{S}_n$-action we have decomposes $V$ into a direct sum as in \eqref{eq:irreducible decomposition}. Therefore, our polynomial $\mathsf{f} \in V$ decomposes as
\begin{align}\label{eq:directsum}
\mathsf{f}=\sum_{\bm{\mu} \vdash n} \sum_{i=1}^{m_{\bm{\mu}}} \mathsf{f}_{\bm{\mu},i}
\end{align}
where $\mathsf{f}_{\bm{\mu},i} \in V_{\bm{\mu}}^i$.
Since $\mathsf{f}$ is $\mathfrak{R}_{\tau_{\lambda}}$-invariant, we may assume without loss of generality that each $\mathsf{f}_{\bm{\mu},i}$ in \eqref{eq:directsum} is also $\mathfrak{R}_{\tau_{\lambda}}$-invariant. Indeed,
$$
\mathsf{f}= \frac{1}{| \mathfrak{R}_{\tau_{\lambda}} |} \sum_{\mathfrak{s} \in \mathfrak{R}_{\tau_{\lambda}}} \mathfrak{s}  \mathsf{f} =
\sum_{\bm{\mu} \vdash n} \sum_{i=1}^{m_{\bm{\mu}}} \left(
\frac{1}{| \mathfrak{R}_{\tau_{\lambda}} |} \sum_{\mathfrak{s} \in \mathfrak{R}_{\tau_{\lambda}}} \mathfrak{s} \mathsf{f}_{\bm{\mu},i} \right),
$$
and since $V_{\bm{\mu}}^i$ is $\mathfrak{R}_{\tau_{\lambda}}$-invariant (because it is $\mathfrak{S}_n$-invariant),
$\frac{1}{| \mathfrak{R}_{\tau_{\lambda}} |} \sum_{\mathfrak{s} \in \mathfrak{R}_{\tau_{\lambda}}} \mathfrak{s}  \mathsf{f}_{\bm{\mu},i}$ lies in $V_{\bm{\mu}}^i$. In the case when $\mathsf{f}$ is some $\mathsf{r}_j$ in the sos \eqref{deepsos}, we are interested in knowing when a $\mathsf{f}_{\bm{\mu},i}$ is non-zero, or equivalently, in determining which parts of $V$ contain $\mathsf{f}$. For this, we rely on the theory of {\em restricted representations}
\cite[Chapter 4]{FultonHarrisBook}.

The restricted representation of the subgroup $\mathfrak{R}_{\tau_{\bm{\lambda}}}$ on $V$ is the
representation of $\mathfrak{R}_{\tau_{\bm{\lambda}}}$ on $V$ obtained by restricting the $\mathfrak{S}_n$ representation $\vartheta \,:\, \mathfrak{S}_n \rightarrow \textup{GL}(V)$ to the elements of $\mathfrak{R}_{\tau_{\bm{\lambda}}}$. This has the effect of refining the
direct sum decomposition in \eqref{eq:irreducible decomposition} since an irreducible $V_{\bm{\mu}}^i$,
while being $\mathfrak{R}_{\tau_{\bm{\lambda}}}$-invariant, may not be irreducible with respect to
$\mathfrak{R}_{\tau_{\bm{\lambda}}}$ and hence will decompose into irreducible representations of
$\mathfrak{R}_{\tau_{\bm{\lambda}}}$. The $\mathfrak{R}_{\tau_{\bm{\lambda}}}$-invariant polynomials in a
$V_{\bm{\mu}}^i$ are contained in the copies of the trivial representation of $\mathfrak{R}_{\tau_{\bm{\lambda}}}$ in $V_{\bm{\mu}}^i$. Therefore, the component $ \mathsf{f}_{\bm{\mu},i}$ in \eqref{eq:directsum} is nonzero only if $V_{\bm{\mu}}^i$ contains at least one copy of the trivial representation of $\mathfrak{R}_{\tau_{\bm{\lambda}}}$ in the restricted representation of $\mathfrak{R}_{\tau_{\bm{\lambda}}}$ on $V_{\bm{\mu}}^i$. To characterize such partitions $\bm{\mu}$, we need the following definition. Let $\bm{\lambda}=(\lambda_1,\lambda_2,\ldots, \lambda_t)$ be a partition and $N_{\bm{\lambda}}$ be the sequence containing $\lambda_p$ copies of the number $p$ for $p =1,\ldots,t$. A semistandard tableau of shape $\bm{\mu}$ and type $\bm{\lambda}$ is a tableau of $\bm{\mu}$ with numbers coming from $N_{\bm{\lambda}}$ such that the numbers are non-decreasing along rows and increasing along columns.

\begin{example}
If $\bm{\lambda}=(4,2,1)$, then $N_{\bm{\lambda}}=(1,1,1,1,2,2,3)$ and the semistandard tableaux of type $\bm{\lambda}$ are

\begin{center}
\young(1111223) \quad \young(111122,3) \quad \young(111123,2)

\vspace{0.2cm}

\quad \young(11112,23) \quad \young(11113,22) \quad \young(11112,2,3)

\vspace{0.2cm}

\quad \young(1111,223)\quad \young(1111,22,3)
\end{center}

Among these, there is one of shape $(7)$, two of shape $(6,1)$, two of shape $(5,2)$, one of shape $(5,1,1)$, one of shape $(4,3)$, and one of shape $(4,2,1)$.
\hfill{\qed}
\end{example}

The following theorem presented in many books (including in Corollary 4.39 and the discussion following it in \cite{FultonHarrisBook}) is exactly what we need.

\begin{theorem}[Young's Rule]\label{YoungsRule}
Let $\bm{\lambda}=(\lambda_1,\ldots, \lambda_t)$ and $\bm{\mu}$ be partitions of $n$. The multiplicity of the trivial representation of $\mathfrak{R}_{\tau_{\bm{\lambda}}}$ in the restriction of some irreducible representation $V^i_{\bm{\mu}}$ to
$\mathfrak{R}_{\tau_{\bm{\lambda}}}$
is equal to the number of semistandard tableaux of shape $\bm{\mu}$ and type $\bm{\lambda}$.
\end{theorem}

Note that in \cite{FultonHarrisBook}, this result is given in terms of {\em induced representations}, but by {\em Frobenius reciprocity} \cite[Corollary 3.20]{FultonHarrisBook}, the above is an equivalent version.

The number of semistandard tableaux of shape $\bm{\mu}$ and type $\bm{\lambda}$ is called  the  {\em Kostka number} $K_{\bm{\mu}\bm{\lambda}}$. We need to know when $K_{\bm{\mu}\bm{\lambda}}$ is non-zero, and
the following standard fact gives a necessary characterization.

\begin{lemma}\label{multzero}
If $\bm{\lambda}$ is lexicographically greater than $\bm{\mu}$ or if $\bm{\mu}$ has more parts than $\bm{\lambda}$, then $K_{\bm{\mu}\bm{\lambda}}=0$.
\end{lemma}

\proof
Suppose $\bm{\lambda}=(\lambda_1, \ldots, \lambda_{t_1})$ and $\bm{\mu}=(\mu_1, \ldots, \mu_{t_2})$. Note that since the numbers are increasing in the columns of any semistandard tableau of shape $\bm{\mu}$ and type $\bm{\lambda}$, the minimum number in row $i$ is $i$.

Suppose $\lambda_i=\mu_i$ for all $1\leq i < i^*$ and $\lambda_{i^*}>\mu_{i^*}$, i.e., $\bm{\lambda}$ is lexicographically greater than $\bm{\mu}$. Then any semistandard tableau of shape $\bm{\mu}$ and type $\bm{\lambda}$ will have to have $\lambda_1$ 1's in row 1, $\ldots$, and $\lambda_{i^*-1}$ $(i^*-1)$'s in row $i^*-1$ in order to have increasing numbers along the columns. Then one would attempt to put $\lambda_{i^*}$ $i^*$'s in row $i^*$, but that is not possible since $\mu_{i^*}<\lambda_{i^*}$, and so at least one $i^*$ will have to be in row $i^*+1$, meaning that we don't have a semistandard tableau since the columns are not strictly increasing.

Now suppose that $\bm{\mu}$ has more parts than $\bm{\lambda}$, i.e., $t_2>t_1$. Look at the first column of a semistandard tableau: there needs to be at least $t_2$ distinct numbers in it, but $N_{\bm{\lambda}}$ contains only $t_1$ distinct numbers and so there is no semistandard tableau.
\qed

\begin{definition} \label{def:unrhd}
We define the ordering $\unrhd$ on partitions as follows:
$\bm{\mu} \unrhd \bm{\lambda}$ if $\bm{\mu}$ is lexicographically greater than or equal to $\bm{\lambda}$ and has at most as many parts as $\bm{\lambda}$.
\end{definition}

Using Theorem \ref{YoungsRule} and Lemma~\ref{multzero}, we can determine which irreducibles in \eqref{eq:irreducible decomposition} contribute components to a given
$\mathfrak{R}_{\tau_{\bm{\lambda}}}$-invariant polynomial.

\begin{theorem} \label{thm:orbit polys in span of basis polys}
A $\mathfrak{R}_{\tau_{\bm{\lambda}}}$-invariant polynomial $\mathsf{f}$ in $V = \RR[\bf x]_{\leq d}$ lies in $\bigoplus_{\bm{\mu} \unrhd \bm{\lambda}} V_{\bm{\mu}}$.
\end{theorem}

\proof
Since $\mathsf{f}$ is $\mathfrak{R}_{\tau_{\bm{\lambda}}}$-invariant, by Theorem \ref{YoungsRule}, the multiplicity of the trivial representation of $\mathfrak{R}_{\tau_{\bm{\lambda}}}$ in $V_{\bm{\mu}}^i$ restricted to $\mathfrak{R}_{\tau_{\bm{\lambda}}}$ is the number of semistandard tableaux of shape $\bm{\mu}$ and type $\bm{\lambda}$. By Lemma \ref{multzero}, this multiplicity is zero if $\bm{\mu}$ is lexicographically smaller than $\bm{\lambda}$ or if it has more parts than $\bm{\lambda}$. Therefore, $\mathsf{f}\in \bigoplus_{\bm{\mu} \unrhd \bm{\lambda}} V_{\bm{\mu}}$. \qed

Since $\mathsf{r}_j$ in $\sum_j \mathsf{r}_j^2$ from \eqref{deepsos}
is $\mathfrak{R}_{\tau_{\bm{\lambda}}}$-invariant for a tableau $\tau_{\bm{\lambda}}$ of shape $(n-k,1^k)$ by Theorem~\ref{lcorbits}, and since for the hook  $\bm{\lambda} = (n-k, 1^k)$, $\bm{\mu} \unrhd \bm{\lambda}$ if and only if $\bm{\mu} \geq_{\textup{lex}} \bm{\lambda}$, we get the following offshot.

\begin{corollary}\label{cor:orbitpolyvectorspace}
Any polynomial $\mathsf{r}_k \in V$ that contributes to the sos \eqref{deepsos} lies in
$\bigoplus_{\bm{\mu} \geq_{\textup{lex}} \bm{\lambda}} V_{\bm{\mu}}$ where
$\bm{\lambda} = (n-k, 1^k)$.
\end{corollary}

\subsection{From Razborov to Gatermann-Parrilo}
We are now ready for the final step. We will show that a sos arising from flag algebras as in Proposition \ref{cor:translation-sos} can be retrieved from the Gatermann-Parrilo SDP by restricting to certain blocks in the sense of
Definition \ref{def:restrictedGP}.  We first define formally the concept of symmetrization,
which we already saw a few times.

\begin{definition}\label{def:symmetrization}
The \emph{symmetrization} of a polynomial $\mathsf{f} \in \RR[{\bf x}]$ with respect to $\mathfrak{S}_n$ is denoted by $\textup{sym}({\mathsf{f}})$ and defined to be

$$\textup{sym}({\mathsf{f}})=\frac{1}{n!}\sum_{\mathfrak{s}\in \mathfrak{S}_n} \mathfrak{s} \mathsf{f}.$$
\end{definition}

\begin{proposition}\label{sosrazvsgp}
 If some $\mathfrak{S}_n$-invariant sos $\sum_{j} \mathsf{f}_j^2$ is such that $\mathsf{f}_j\in V_{\bm{\lambda}_1}\oplus \ldots\oplus V_{\bm{\lambda}_s}$ for all $j$, then this sos can be obtained through the Gatermann-Parrilo SDP restricted to  $\bm{\lambda}_1, \ldots, \bm{\lambda}_s$.
\end{proposition}

\proof
Suppose $\sum_{j}\mathsf{f}_j^2$ is a sos such that
$\mathsf{f}_j=\mathsf{f}_{j,1} + \cdots + \mathsf{f}_{j,s}$ where
$\mathsf{f}_{j,i} \in V_{{\bm{\lambda}}_i}$ for $i=1,\ldots,s$ for all $j$.
Then
\begin{align} \label{expansion of rk^2}
\mathsf{f}_j^2= (\mathsf{f}_{j,1})^2+\ldots + (\mathsf{f}_{j,s})^2 +
2\sum_{i_1 < i_2} \mathsf{f}_{j,i_1} \mathsf{f}_{j,i_2}.
\end{align}

Since the sos $\sum_{j}\mathsf{f}_j^2$ is $\mathfrak{S}_n$-invariant, symmetrizing it with respect to $\mathfrak{S}_n$ leaves it unchanged. Therefore,
\begin{align*}
\sum_{j}\mathsf{f}_j^2&=\textup{sym}\left(\sum_{j}\mathsf{f}_j^2\right)=\sum_{j} \textup{sym} \left(\mathsf{f}_j^2\right)=\sum_{j} \textup{sym}\left((\mathsf{f}_{j,1}+\ldots +\mathsf{f}_{j,s})^2\right)\\
&=\sum_{j} \left(\sum_{i} \textup{sym} \left(\mathsf{f}_{j,i}^2\right) + 2\sum_{i_1<i_2} \textup{sym}\left(\mathsf{f}_{j,i_1}\mathsf{f}_{j,i_2}\right)\right).
\end{align*}

Note that $\mathsf{f}_{j,i_1}\mathsf{f}_{j,i_2} \in V_{\bm{\lambda}_{i_1}} \otimes V_{\bm{\lambda}_{i_2}}$. By Exercise 4.51 in \cite{FultonHarrisBook}, the trivial representation of $\mathfrak{S}_n$ is present in $V_{\bm{\lambda}_{i_1}}\otimes V_{\bm{\lambda}_{i_2}}$ once if $i_1=i_2$ and zero times otherwise.
 For $i_1 \neq i_2$, $\textup{sym}(\mathsf{f}_{j,i_1}\mathsf{f}_{j,i_2})$ is $\mathfrak{S}_n$-invariant as well by symmetrization and therefore, must be the zero vector.

Thus, $\mathsf{f}_j^2=\mathsf{f}_{j,1}^2+\ldots + \mathsf{f}_{j,s}^2$, so each square in $\sum_{j}\mathsf{f}_j^2$ can be obtained from restricting the Gatermann-Parrilo SDP to partitions $\bm{\lambda}_1, \ldots, \bm{\lambda}_s$.

\qed

\begin{theorem} \label{cor:sos from same isotypics}
The sos \eqref{middlesos} can be obtained by restricting the Gatermann-Parrilo SDP to the partitions $\bm{\mu}$ where $\bm{\mu}\geq_{\textup{lex}} (n-k, 1^k)$.
\end{theorem}

\proof
First note that the sos $\sum_{\theta}\sum_{j} \mathsf{r}_{\theta,j}^2$ is invariant under $\mathfrak{S}_n$ since we are taking the expectation over all maps $\theta$. Moreover, $\mathsf{r}_{\theta,j}$ is invariant under $\mathfrak{R}_{\tau_{\bm{\lambda}}}$ for some tableau $\tau_{\bm{\lambda}}$ of shape $(n-k,1^k)$ for all $\theta$ and $j$. Indeed, even though different $\theta$'s will require different tableaux, all will be of shape $(n-k,1^k)$. Thus all $\mathsf{r}_{\theta,k}$ lie in $\bigoplus_{\bm{\mu} >_{\textup{lex}} (n-t,1^t)}V_{\bm{\mu}}$ by \eqref{cor:orbitpolyvectorspace}. Therefore, by \eqref{sosrazvsgp}, the sos \eqref{middlesos} can be obtained from the Gatermann-Parrilo SDP restricted to the partitions $\bm{\mu} \geq_{\textup{lex}} \bm{\lambda}$.
\qed

Finally, we tackle the general Razborov sos  for proofs involving several $\sigma$ and $l$.

\begin{corollary}\label{wholething}
Consider a conic combination of sos \eqref{razborovsosshape} for different choices of $\sigma$ and $l$. Let $k^*$ be the maximum size of all types $\sigma$ present.
Then this sos can be obtained through the Gatermann-Parrilo program restricted to partitions lexicographically greater or equal to the hook partition $(n-k^*, 1^{k^*})$.
\end{corollary}

\proof
In Theorem \ref{cor:sos from same isotypics}, we proved that the sos
\eqref{deepsos} can be obtained by restricting the Gatermann-Parrilo SDP to partitions $\bm{\mu}$ for $\bm{\mu}\geq_{\textup{lex}} (n-|\sigma|,1^{|\sigma|})$. Thus every sos in \eqref{razborovsosshape} for each $\theta, \sigma, l$ can be obtained from partitions $\bm{\mu}$ for $\bm{\mu}\geq_{\textup{lex}} (n-k^*, 1^{k^*})$ where $k^*$ is the maximum size of all types considered. Since the final sos is a conic combination of these smaller sos, it can also be obtained by restricting the Gatermann-Parrilo SDP to those same partitions.
\qed

We have established the relationship between the Gatermann-Parrilo framework and the flag algebra sos for Tur{\'a}n problems.  We conclude with a few remarks that will be helpful in implementing the Gatermann-Parrilo framework.

\begin{remark}
Note that flag algebra Tur\'an sos uses fixed $k$ and $l$ (independent
of $n$) and thus the number of partitions indexing blocks in the
restricted SDP is also independent of $n$. Indeed, the number of
partitions lexicographically greater than or equal to $(n-k,1^{k})$ is
at most twice the number of partitions of $k$. Moreover, each small
subblock in a block corresponding to a $\bm{\lambda}$ has
size $m_{\bm{\lambda}}\times m_{\bm{\lambda}}$, which depends on $d$
but not $n$. Finally, even though the number of subblocks per block
(i.e., the number of standard tableaux) increases as $n$ increases, by
\eqref{onlyonetableau}, we only need one subblock. Thus, the size of
the restricted SDP does not depend on $n$.
\end{remark}

\begin{remark}
In Theorem~\ref{lcorbits}, we showed that a $\mathsf{r}_j$ in the sos
\eqref{deepsos} is invariant under
$\mathfrak{R}_{\tau_{\bm{\lambda}}}$ for some tableau
$\tau_{\bm{\lambda}}$ where $\bm{\lambda}$ is a hook partition. One
can obtain further savings in the size of the corresponding SDP by
using symmetries of the $\sigma$-flags which allows us to replace
the hook partition with partitions that are lexicographically larger
than it.  For example, in Example~\ref{exorbit_part2},
$\mathsf{p}_F^\theta$ is also invariant under the row group of
$$\tau_{\bm{\lambda}} = \young(45,23,1) \,\,\,\,\textup{, i.e., }\,\,\,\,
\mathfrak{R}_{\tau_{\bm{\lambda}}} = \{1, (2,3), (4,5), (2,3)(4,5)\}.$$
\end{remark}

We also note that we have stated our results based on the isotypic decomposition of $V = \RR[{\mathbf x}]_{\leq d}$ and not of the quotient vector space $\RR[{\mathbf x}]_{\leq d} / \mathscr{I}_n^A$. This is again for simplicity and will not change the flavor of the results. The ideal is used implicitly in some places such as in the derivation of the sos \eqref{razborovsosshape}.

\section{The Symmetry-Adapted SDP for Mantel's Theorem} \label{sec:Mantel1}
Recall the proof of Mantel's theorem from Section~\ref{sec:FlagAlgebras}  using flag algebra calculus (see also the Appendix). The sos was divided into two parts: one involving $F_0$ and $F_1$, and one involving $H_1$. We show how to retrieve the former; the latter is similar in flavor.
Since the flags $F_0$ and $F_1$ needed in the sos expression~\eqref{eq:mantel-short-sos} have a type of size $k=1$, Theorem~\ref{cor:sos from same isotypics} implies that we can retrieve this sos expression by restricting the Gatermann-Parrilo SDP to partitions $(n)$ and $(n-1,1)$. Moreover,
 $\mathsf{p}_{F}^{\theta}$ has degree at most one for $F\in \{F_0, F_1\}$ and therefore we only need to consider $V=\RR[{\mathbf{x}}]_{\leq 1}$. For illustration, we verify this.

  Observe that $n_{\bm{\lambda}}$, the number of standard tableaux of shape $\bm{\lambda}$, is $1$ for $\bm{\lambda}=(n)$ and $n-1$ for $\bm{\lambda}=(n-1,1)$. Restricting the expression in \eqref{onlyonetableau} to partitions $(n)$ and $(n-1,1)$ implies that there exist psd matrices $Q_{(n)}$ and $Q_{(n-1,1)}$ such that
\begin{align}\label{eq:sos-mantel-short2}
\mathsf{f} = 1 \cdot \langle  Q_{(n)}, Y_{(n)}\rangle + (n-1)\cdot \langle Q_{(n-1,1)}, Y_{(n-1,1)} \rangle
\end{align}

 \noindent where
 \begin{align*}
\mathsf{f}=E_{\theta}\left[\begin{pmatrix}\pp^{\theta}_{F_0}(G) & \pp^{\theta}_{F_1}(G)\end{pmatrix}
\begin{pmatrix} \frac12   &-\frac12 \\ -\frac12 &  \frac12\end{pmatrix} \begin{pmatrix}\pp^{\theta}_{F_0}(G) \\ \pp^{\theta}_{F_1}(G) \end{pmatrix}\right]
\end{align*}
is the sos expression obtained using flag algebras in \eqref{eq:mantel-short-sos}.
Recall that the expression~\eqref{eq:sos-mantel-short2} relies on the decomposition  $V_{\bm{\lambda}} = \oplus_{i=1}^{n_{\bm{\lambda}}} W_{\tau_{\bm{\lambda}}^i}$ as in Equation~\eqref{eq:alternate-decomposition}.  The matrices $Y_{\bm{\lambda}}$ come from the polynomials in the symmetry-adapted basis and can be computed using the algorithm in \cite[Chapter 5.2, pp 113-114]{FaesslerStiefelBook}. Also recall that $Y_{\bm{\lambda}}$ depends only on one standard tableau of shape $\bm{\lambda}$.

 The algorithm yields the basis polynomials
 \begin{align*}
\mathsf{p}_{0,0}:=1, \quad \mathsf{p}_{0,1}:= \frac{1}{\sqrt{\binom{n}{2}}}\cdot \sum_{1\leq i<j\leq n} \mathsf{x}_{ij}
\end{align*}
of degrees zero and one for $V_{(n)}=W_{\tau_{(n)}}$ where $\tau_{(n)}=\young(12<\ldots>n)$.
Therefore,
$$Y_{(n)}=\begin{pmatrix}
\mathsf{p}_{0,0} & \mathsf{p}_{0,1}
\end{pmatrix}
\begin{pmatrix}
\mathsf{p}_{0,0} \\
\mathsf{p}_{0,1}
\end{pmatrix}.$$

Similarly, the algorithm yields the basis polynomial

$$p_{1,n}:=\frac{1}{\sqrt{\binom{n-1}{2}\cdot 1^2+(n-1)\cdot \left(\frac{n-2}{2}\right)^2}}\left(\sum_{1\leq i < j \leq n-1}\mathsf{x}_{ij} - \frac{n-2}{2}\cdot\sum_{1\leq i \leq n-1} \mathsf{x}_{in}\right)$$
of degree one for $W_{\tau_{(n-1,1)}}$ where $\tau_{(n-1,1)}=\Yboxdimx{25pt}\Yboxdimy{10pt}\young(123<\ldots><n-1>,n).$
Thus, we have that
$$Y_{(n-1,1)}=\frac{1}{|\mathfrak{S}_n|} \sum_{\mathfrak{s}\in \mathfrak{S}_n} \mathfrak{s}(\mathsf{p}_{1,n})(\mathsf{p}_{1,n})^T=\frac{1}{n}\sum_{i=1}^n (\mathsf{p}_{1,i})^2.$$

 Altogether, we obtain a block SDP consisting of one block of size $2\times 2$ for $Q_{(n)}$ and another block of size $1\times 1$ for $Q_{(n-1,1)}$. For instance choosing psd matrices
$$Q_{(n)}=\begin{pmatrix}
\frac{(n-1)^2}{2} & \frac{-\sqrt{2(n-1)^3}}{\sqrt{n}} \\
\frac{-\sqrt{2(n-1)^3}}{\sqrt{n}}  & \frac{4(n-1)}{n}
\end{pmatrix}$$
 and

$$Q_{(n-1,1)}=\begin{pmatrix}
\frac{2(n-1)(n-2)}{n}
\end{pmatrix}$$
 yields the expression~\eqref{eq:mantel-short-sos} as desired.

\section{Conclusion}\label{sec:conclusion}
The main result of this paper is that standard symmetry-reduction methods in polynomial sos theory can retrieve
Razborov's sos proofs using flag densities that arise in the context of Tur\'an problems.  For the sake of notational simplicity, we presented our results in the context of graphs. However, the same techniques can be used for Tur\'an and non-Tur\'an problems over hypergraphs, digraphs, tournaments, etc. Below, we give a few examples of the changes that need to be made in different cases.

\begin{enumerate}
\item {\em Hypergraphs}: For any optimization problem over $a$-uniform hypergraphs on $n$ vertices, we can use the ideal

$$\langle  x_{i_1\ldots i_a}^2-x_{i_1\ldots i_a} \quad  \forall 1\leq i_1 < \ldots < i_a \leq n \rangle. $$

This ideal is $\mathfrak{S}_n$-invariant, which implies that our theory can still be applied. Thus any flag Cauchy-Schwarz proof in this setting can be retrieved by using symmetry-reduction and restricting to partitions lexicographically greater or equal to $(n-k^*, k^*)$ where $k^*$ is the size of the biggest $\sigma$ involved and by letting the degree be as big as the biggest flag present.

For example, we could use this for the Tur\'an hypergraph problem of maximizing the hyperedge density in an $a$-uniform hypergraph on $n$ vertices where $a$-uniform hypercliques of size $b$ are forbidden by adding 

$$\prod_{\{i_1, \ldots, i_a\} \in E(Q)} x_{i_1\ldots i_a} \quad  \forall b\textup{-clique } Q \textup{ in } K_n^a$$
to our ideal.  In \cite{Razborov10}, Razborov gives a Cauchy-Schwarz proof that the maximum edge density in 
a $3$-uniform hypergraph without $4$-cliques is at most $0.561666$. This result can be translated in the 
polynomial language using the above-mentioned framework.

\item {\em Digraphs}: For any optimization problem over digraphs on $n$ vertices, we can use the ideal

$$\langle x_{ij}^2-x_{ij} \quad  \forall 1\leq i,j \leq n \textup{ such that } i\neq j\rangle. $$
Again, the ideal is $\mathfrak{S}_n$-invariant, and so our techniques can again be used here.

For example, we could use this for the Cacetta-H\"aggkvist conjecture which states that every simple digraph of order $n$ with minimum outdegree of at least $r$ has a cycle of length at most $\lceil \frac{n}{r} \rceil$ by adding

$$\prod_{(i,j)\in E(C)} x_{ij}  \quad \forall \textup{ cycles } C \textup{ such that } |C|\leq \lceil\frac{n}{r} \rceil$$
$$ \prod_{j\in S} (1-x_{ij}) \quad \forall i, \forall S \subseteq [n]\backslash \{i\} \textup{ such that } |S|=n-r+1$$
and check for feasibility. Indeed, the first set of constraints ensures that we get a directed graph on $n$ vertices, the second set of constraints forbids cycles of length less or equal to $\lceil \frac{n}{r} \rceil$, and the third set of constraints sets the outdegree of every vertex to be t least $r$. Of course, here the degree of these constraints being quite high makes it unlikely that this technique would yield any interesting results unless $r$ is very big.

\item {\em Tournaments}:  For an optimization problem over all tournaments of size $n$, we can use the ideal

\begin{align*}
\langle & x^2_{ij}-x_{ij} && \forall 1 \leq i,j \leq n \textup{ such that } i\neq j\\
&x_{ij}x_{ji} && \forall 1 \leq i < j \leq n \\
&(1-x_{ij})(1-x_{ji}) && \forall 1 \leq i < j \leq n \rangle
\end{align*}

since exactly one of the arcs $(i,j)$ and $(j,i)$ are present for any $1 \leq i < j \leq n$, as well as any additional constraints forbidding certain structures in the problem. Again, we have an ideal that is $\mathfrak{S}_n$-invariant, and our techniques can again be used here.
\end{enumerate}

The general theory of flag algebras \cite{RazborovFlagAlgebras} is couched in a broad setting using tools from algebra, topology, and probability. It is a formal calculus that came about from an attempt to systematize and distill the many ad hoc methods in extremal combinatorics. In this paper, we have shown that sos proofs that arise from flag densities are equivalent to polynomial sos proofs using symmetry-reduction. This connection allows for systematic ways to search for flag sos proofs which was part if our original motivation in undertaking this project.  There are likely much further, and  deeper, connections between Razborov's theory in its full generality, and the fundamental structures of real algebraic geometry. We hope to delve deeper in this direction.

In \cite{RSST}, the authors, along with James Saunderson, established the converse to the result in this paper. By a $k$-subset hypercube we mean a hypercube whose coordinates are indexed by the $k$-element subsets of $[n]$. Thus the usual hypercube $\{0,1\}^n$ is the $1$-subset hypercube. The $2$-subset hypercube $\{0,1\}^{n \choose 2}$ arises in the context of optimization over (the edges of) a graph. The main result of \cite{RSST} is that flag methods can be used to provide sos certificates for the nonnegativity of symmetric polynomials over $k$-subset hypercubes. This extends their use beyond the realm of extremal combinatorics into general polynomial optimization. The two papers together establish that flag methods are equivalent to standard symmetry-reduction methods in polynomial sos theory over the class of  $k$-subset hypercubes.

\section{Appendix}
We present here the full translation of \cite{Falgas-RavryVaughan} to polynomials, culminating with a verification of the equivalence~\eqref{equivalencemantel}. We first discuss the (non-trivial) inequality whose nonnegativity establishes the theorem. A few other quantities are also needed for the proof in \cite{Falgas-RavryVaughan} which we define here.

Fix $m<n$, and consider $H \in \mathcal{G}_m$. Then we obtain the following equivalences:

\begin{align}
\mathsf{p}_n &\equiv \sum_{H \in \mathcal{G}_m} \mathsf{p}_H\mathsf{p}_m(\mathds{1}_H) \label{eq:subgraph density facts}\  \mod \mathscr{I}_n^{{K_3}},\\
1 &\equiv \sum_{H \in \mathcal{G}_m} \mathsf{p}_H \ \mod \mathscr{I}_n^{{K_3}}.\label{eq:subgraph density facts 2}
\end{align}
Both equivalences follow from the fact that two polynomials are equivalent mod $\mathscr{I}_n^{{K_3}}$ if and only if they have the same value on $\mathds{1}_G$  for each $G \in \mathcal{G}_n$. The first equivalence is since the
edge density of $G$ is the sum of the edge densities of $H \in \mathcal{G}_m$ weighted by the density of $H$ in $G$.
The second is since $\mathsf{p}_H(\mathds{1}_G)$ is the probability of $H$ in $G$.

Note that $\left(\mathsf{q}^\alpha_H\right)^2 \equiv \mathsf{q}^\alpha_H \mod \mathscr{I}_n^{K_3}$ since
$\mathsf{x}_{ij}$ and $(1-\mathsf{x}_{ij})$ are equivalent to their squares mod $\mathscr{I}_n^{K_3}$, which implies that
\begin{align}\label{sosdensity}
\mathsf{p}_H\equiv \frac{1}{a_H \binom{n}{|V(H)|}}\sum_{\alpha\in \textup{Inj}(V(H),[n])}\left(\mathsf{q}^\alpha_H\right)^2\    \mod \mathscr{I}_n^{K_3},
\end{align}
and thus $\mathsf{p}_H$ is a sos mod $\mathscr{I}_n^{K_3}$, a fact that will be useful later. Multiplying
$\max_{H\in \mathcal{G}_m} \mathsf{p}_m(\mathds{1}_H)$ by $1$ and then using equations~\eqref{eq:subgraph density facts} and ~\eqref{eq:subgraph density facts 2}, we get that

\begin{align} \label{eq:subgraph_density_inequality0}
\max_{H\in \mathcal{G}_m} \mathsf{p}_m(\mathds{1}_H) -\mathsf{p}_n
&\equiv  \sum_{H \in \mathcal{G}_m} \left( \underbrace{\max_{H'\in \mathcal{G}_m} \mathsf{p}_m(\mathds{1}_{H'})- \mathsf{p}_m(\mathds{1}_H)}_{\geq 0}\right) \mathsf{p}_H     \ \mod \mathscr{I}_n^{{K_3}}.
\end{align}
By \eqref{sosdensity}, the right-hand side of the above expression is a sos
polynomial mod $\mathscr{I}_n^{K_3}$ , and hence, $\max_{H\in \mathcal{G}_m} \mathsf{p}_m(\mathds{1}_H)$ is an upper bound on the edge density of any $G \in \mathcal{G}_n$. Typically, this bound does not give a sufficiently tight answer.

Now suppose we also had a sos polynomial mod $\mathscr{I}_n^{K_3}$ of the following type:
\begin{align} \label{eq:subgraph density inequality}
 \sum_{H \in \mathcal{G}_m}  c_H\mathsf{p}_H \equiv \sum_{i} \mathsf{s}_i^2 \ \mod \mathscr{I}_n^{\mathcal{K}_3}
\end{align}
where $c_H\in \mathbb{R}$ for each $H\in \mathcal{G}_m$.
Then adding inequality \eqref{eq:subgraph density inequality} to \eqref{eq:subgraph density facts}
and by a similar calculation to that in  \eqref{eq:subgraph_density_inequality0}, we get
\begin{align*}
&\max_{H \in \mathcal{G}_m} (\mathsf{p}_m(\mathds{1}_H) + c_H) -\mathsf{p}_n \equiv \\
 & \sum_{H \in \mathcal{G}_m} \left( \underbrace{\left(\max_{H'\in \mathcal{G}_m} \mathsf{p}_m(\mathds{1}_{H'})+c_{H'}\right)- \left(\mathsf{p}_m(\mathds{1}_H)+c_H\right)}_{\geq 0}\right) \mathsf{p}_H  +\sum_k \mathsf{s}_k^2 \mod \mathscr{I}_n^{{K_3}}.
\end{align*}
Again, since the right-hand side of the equivalence above is sos, we obtain the improved bound of
\begin{align}\label{eq:flag algebra bound}
\max_{H \in \mathcal{G}_m} (\mathsf{p}_m(\mathds{1}_H) + c_H)
\end{align}
on the edge density of any $G \in \mathcal{G}_n$. So the goal becomes to
find equivalences as in \eqref{eq:subgraph density inequality} which then feeds into \eqref{eq:flag algebra bound} to yield the desired bound of $\frac{1}{2}+O(\frac{1}{n})$. We now explain how this is done using flag algebras.

We must first define the density of two flags.
Fix a type $\sigma$ of size $k$ and $l\geq k$.  For a single flag
$F \in \mathcal{F}^\sigma_l$ and a fixed labeling $\theta$ of $k$ vertices in $G$ using all the labels in $[k]$, we have the density polynomial $\mathsf{p}^\theta_F$. Recall that $\mathsf{p}^\theta_F(\mathds{1}_G)$ is the probability that the $k$ vertices of $G$ labeled by $\theta$ along with $l-k$ unlabeled vertices picked uniformly at random induce a copy of $F$ in $G$. Now continue to select a disjoint set of $l'-k$ unlabeled vertices in $G$
uniformly at random. Then given a second $\sigma$-flag  $F' \in \mathcal{F}^\sigma_{l'}$, $\mathsf{p}_{F,F'}^\theta(\mathds{1}_G)$ is the probability that the two sets of unlabeled vertices each taken separately with
the graph induced by $\theta$ induce copies of $F$ and $F'$ respectively, where

\begin{align*}
\mathsf{p}^{\theta}_{F,F'} & = \frac{1}{\binom{n-k}{l-k}\binom{n-l}{l'-k}} \sum_{\substack{S,S'\subseteq[n]:\\|S|=l, |S'|=l'\\S\cap S'=\textup{im}(\theta)}}\left(\frac{1}{a_F^\sigma} \sum_{\alpha\in \textup{Inj}_\theta(V(F),S)} \mathsf{q}_F^\alpha \right)\left(\frac{1}{a_{F'}^\sigma} \sum_{\alpha\in \textup{Inj}_\theta(V(F'),S')} \mathsf{q}_{F'}^\alpha \right)\\
& = \frac{1}{a_F^\sigma a_{F'}^\sigma \binom{n-k}{l-k}\binom{n-l}{l'-k}} \sum_{\substack{\alpha_1 \in \textup{Inj}_\theta(V(F),[n])\\ \alpha_2 \in\textup{Inj}_{\theta}(V(F'),[n]\backslash \alpha_1(V(F\backslash \sigma)))}} \mathsf{q}_F^{\alpha_1} \mathsf{q}_{F'}^{\alpha_2}.
\end{align*}

Note that $\mathsf{p}^\theta_{F,F'}(\mathds{1}_G)=p(F_1,F_2;F_3)$ in \cite{RazborovFlagAlgebras} where $F_1=(F,\theta)\in \mathcal{F}_l^{\sigma}$, $F_2=(F',\theta)\in \mathcal{F}_{l'}^{\sigma}$ and $F_3=(G,\theta)\in \mathcal{F}_n^\sigma$ where $|\sigma|=k$.

Next, let $\mathsf{p}_F:\equiv \mathbb{E}_\theta \mathsf{p}^\theta_{F} \ \mod \mathscr{I}_n^{{K_3}}$  where $\mathsf{p}_{F}$ starts by choosing a map $\theta:[k]\rightarrow [n]$ uniformly at random and then computes $\mathsf{p}^\theta_{F}$ for it. Similary, let $\mathsf{p}_{F,F'} :\equiv \mathbb{E}_\theta \mathsf{p}^\theta_{F,F'} \ \mod \mathscr{I}_n^{{K_3}}$ where $\mathsf{p}_{F,F'}$ starts by choosing a map $\theta:[k]\rightarrow [n]$ uniformly at random and then computes $\mathsf{p}^\theta_{F,F'}$ for it.

If $k$, $l$ and $l'$ are fixed and $n$ is large, then picking two random extensions of
$\theta$ of sizes $l-k$ and $l'-k$ is essentially the same as picking two disjoint extensions. Indeed, the probability that the
two random extensions share  $d$ vertices is roughly $\frac{d}{n}$ where $d \leq \max\{l,l'\}-k$. Therefore, the probability of the two
random extensions overlapping is $O\left(\frac{1}{n}\right)$ where we hide factors depending on $k$ and $\max\{l,l'\}$.
This implies that for all partial labelings $\theta: [k]\rightarrow [n]$ there is a polynomial $\mathsf{err}^{\theta}_{F,F'}(\mathbf{x})$ such that
\begin{align} \label{eq:probabilities}
\mathsf{p}^\theta_F \mathsf{p}^\theta_{F'} \equiv \mathsf{p}^\theta_{F,F'} + \mathsf{err}^{\theta}_{F,F'}  \mod \mathscr{I}_n^{K_3}, \textrm{ and}\\
\mathsf{err}^{\theta}_{F,F'}\equiv  O\left(\frac1n\right)\mod \mathscr{I}_n^{K_3}.
\end{align}

Taking expectation with respect to the uniform distribution over all possible partial labelings $\theta$, and letting $\mathsf{err}_{F,F'}(\mathbf{x})=\mathbb{E}_{\theta}\,\mathsf{err}_{F,F'}^{\theta}(\mathbf{x})$, we obtain that
\begin{align} \label{eq:probabilities2}
\mathbb{E}_\theta \,\mathsf{p}^\theta_F \mathsf{p}^\theta_{F'} \equiv \mathsf{p}_{F,F'} +\mathsf{err}_{F,F'} \ \mod \mathscr{I}_n^{K_3}
\end{align}
where again $\mathsf{err}_{F,F'}\equiv  O\left(\frac1n\right)\mod \mathscr{I}_n^{K_3}$.

Going back to $m$-vertex graphs $H \in \mathcal{G}_m$, if $m \geq l+l'-k$, then by the same reasoning as in \eqref{eq:subgraph density facts}, we also have
\begin{align}\label{eq:flagsandH}
\mathsf{p}_{F,F'}\equiv  \sum_{H \in \mathcal{G}_m} \mathsf{p}_H \mathsf{p}_{F,F'}(\mathds{1}_H) \ \mod \mathscr{I}_n^{K_3}.
\end{align}
We can now use these observations to produce an expression of the form \eqref{eq:subgraph density inequality} as follows. For any psd matrix  $Q$ of size $|\mathcal{F}^\sigma_l| \times |\mathcal{F}^\sigma_l|$, the polynomial
$\sum_{F,F' \in \mathcal{F}^\sigma_l} Q_{F,F'} \mathsf{p}^\theta_F \mathsf{p}^\theta_{F'} $, and
hence, $\mathbb{E}_\theta \sum_{F,F' \in \mathcal{F}^\sigma_l} Q_{F,F'} \mathsf{p}^\theta_F \mathsf{p}^\theta_{F'}$, are
sos mod $\mathscr{I}_n^{K_3}$. Now note that
\begin{align*}
   \mathbb{E}_\theta \sum_{F,F' \in \mathcal{F}^\sigma_l} Q_{F,F'} \mathsf{p}^\theta_F \mathsf{p}^\theta_{F'}   & \equiv \sum_{F,F' \in \mathcal{F}^\sigma_l} Q_{F,F'} \left(\mathsf{p}_{F,F'}  + \mathsf{err}_{F,F'}\right) \mod \mathscr{I}_n^{K_3}\\
&\equiv \sum_{F,F' \in \mathcal{F}^\sigma_l} Q_{F,F'} \mathsf{p}_{F,F'}  + \mathsf{err}\ \mod \mathscr{I}_n^{K_3}\\
   & \equiv  \sum_{F,F' \in \mathcal{F}^\sigma_l} Q_{F,F'}  \left(\sum_{H \in \mathcal{G}_m} \mathsf{p}_H \mathsf{p}_{F,F'}(\mathds{1}_H) \right) + \mathsf{err} \ \mod \mathscr{I}_n^{K_3}\\
& \equiv  \sum_{H \in \mathcal{G}_m} \mathsf{p}_H \underbrace{\left(\sum_{F,F' \in \mathcal{F}^\sigma_l} Q_{F,F'} \mathsf{p}_{F,F'}(\mathds{1}_H) \right)}_{c_H} + \mathsf{err}\ \mod \mathscr{I}_n^{K_3}
\end{align*}
where $\mathsf{err} (\mathds{1}_G)$ has value $O\left(\frac1n\right)$ for every $G\in \mathcal{G}_n$, i.e., on the zeros of $\mathscr{I}_n^{K_3}$.
Therefore, just as we derived the bound on edge density following \eqref{eq:subgraph density inequality}, we get that
modulo $\mathscr{I}_n^{{K_3}}$,
\begin{align*}
&\max_{H \in \mathcal{G}_m} (\mathsf{p}_m(\mathds{1}_H) + c_H) -\mathsf{p}_n +\mathsf{err} \equiv \\
& \sum_{H \in \mathcal{G}_m} \left(\max_{H'\in \mathcal{G}_m} \left(\mathsf{p}_m(\mathds{1}_{H'})+c_{H'}\right)- \left(\mathsf{p}_m(\mathds{1}_H)+c_H\right)\right) \mathsf{p}_H  + \mathbb{E}_\theta \sum_{F,F' \in \mathcal{F}^\sigma_l} Q_{F,F'} \mathsf{p}^\theta_F \mathsf{p}^\theta_{F'}.
\end{align*}
Since the right-hand side is a sos expression for any graph $G\in \mathcal{G}_n$, we have $\mathsf{p}_n(\mathds{1}_G)\leq \mathsf{p}_m(\mathds{1}_H) + c_H+O\left(\frac1n\right)$ which is an upper bound on the edge density of $G$.

In summary, to apply the flag Cauchy-Schwarz approach to the Mantel problem, one chooses parameters $k,l$ and $m$ appropriately. Then the choice of the psd matrix $Q$ that obtains the best possible bound on the triangle-free density density as in \eqref{eq:flag algebra bound} reduces to solving a semidefinite program whose size depends on $k,l$ and $m$ and is independent of $n$. Note that in certain proofs, one must choose several $k,l,m$ and take a linear combination of \eqref{eq:subgraph density inequality} to obtain the desired upper bound; this however doesn't change the explanations above in a significant way.

We now illustrate the above procedure on the Mantel example.
Set $k = 1$ and take $\sigma$ to be a  single vertex labeled ``1''. Taking $l = 2$, we have the set of flags
$\mathcal{F}^\sigma_l = \{ F_0, F_1 \}$ where
$$ F_0 = \labeledvnonedge{}{1}, \,\,\,\, F_1 = \labeledvedge{}{1}.$$
Now we choose $m=3$ so that $\mathcal{H}$ is the set of all triangle-free graphs on three vertices. Up to isomorphism, there is a
unique graph on three vertices with $i$ edges for $i=0,1,2,3$. Call these $H_0, H_1$, $H_2$, and $H_3$, i.e.,
$$H_0=\Hzero, \quad H_1=\Hone, \quad H_2=\Htwo, \quad H_3=\Hthree.$$
Therefore, $\mathcal{G}_3 =\{H_0,H_1,H_2\}$ since $H_3$ is not triangle-free.

The density polynomials $\mathsf{p}_{H_i}$ for each $0\leq i\leq 3$ can be defined according to equation~\eqref{eq:dh}. For example, we have $\mathsf{p}_{H_0}= \frac{1}{{n \choose 3}}\sum_{i<j<k} (1-\mathsf{x}_{ij})(1-\mathsf{x}_{ik})(1-\mathsf{x}_{jk}).$

Using the fact that $\mathsf{p}_{H_3} \equiv 0 \textrm{ mod } \mathscr{I}_n^{K_3}$, we rewrite equations \eqref{eq:subgraph density facts} and \eqref{eq:subgraph density facts 2} as:
\begin{align*}
\mathsf{p}_n&\equiv \frac13 \mathsf{p_{H_1}}+ \frac23\mathsf{p_{H_2}} \textrm{ mod }
\mathscr{I}_n^{K_3},\\
1&\equiv  \mathsf{p_{H_0}}+ \mathsf{p_{H_1}}+ \mathsf{p_{H_2}}\textrm{ mod } \mathscr{I}_n^{K_3}.
\end{align*}
We now calculate $\mathsf{p}_{F_i,F_j}(\mathds{1}_{H_k})$ for every $i,j\in \{0,1\}$ and $k\in \{0,1,2,3\}$  and use \eqref{eq:flagsandH} to obtain the following relations

\begin{eqnarray*}
\pp_{F_0,F_0} &=& \pp_{H_0}+ \frac13 \pp_{H_1},\\
\pp_{F_0,F_1} =\pp_{F_1,F_0}  &=&  \frac 13 \pp_{H_1}+ \frac 13 \pp_{H_2},\\
\pp_{F_1,F_1} &=& \frac13 \pp_{H_2}+ \pp_{H_3} \equiv \frac13 \pp_{H_2} \textrm{ mod } \mathscr{I}_n^{K_3}.
\end{eqnarray*}
Choosing $Q= \begin{pmatrix} \frac12   &-\frac12 \\ -\frac12 &\frac12\end{pmatrix}$, we obtain
the following polynomial identity

\begin{align}\label{eq:mantel-short-sos}
& \frac12 \pp_{H_0}- \frac16 \pp_{H_1}- \frac16\pp_{H_2}+\mathsf{err}
\equiv\mathbb{E}_{\theta}\left[\begin{pmatrix}\mathsf{p}^{\theta}_{F_0} & \mathsf{p}^{\theta}_{F_1}\end{pmatrix}  \begin{pmatrix} \frac12   &-\frac12 \\ -\frac12 &\frac12\end{pmatrix} \begin{pmatrix}\mathsf{p}^{\theta}_{F_0} \\ \mathsf{p}^{\theta}_{F_1} \end{pmatrix}\right]
\end{align}
as in the sos expression ~\eqref{eq:subgraph density inequality}.  Thus we get the following proof:

\begin{align*}
&\frac12-\frac{1}{{n \choose 2}}\sum_{1\leq i< j\leq n} \mathsf{x}_{ij} + \mathsf{err}   \equiv \frac12 \pp_{H_0}- \frac16 \pp_{H_1}- \frac16\pp_{H_2}+\frac13 \pp_{H_1}+ \mathsf{err}  \textrm{ mod } \mathscr{I}_n^{K_3}\\
&\quad \quad \equiv\mathbb{E}_{\theta}\left[\begin{pmatrix}\mathsf{p}^{\theta}_{F_0} & \mathsf{p}^{\theta}_{F_1}\end{pmatrix}  \begin{pmatrix} \frac12   &-\frac12 \\ -\frac12 &\frac12\end{pmatrix} \begin{pmatrix}\mathsf{p}^{\theta}_{F_0} \\ \mathsf{p}^{\theta}_{F_1} \end{pmatrix}\right] +\frac13 \pp_{H_1}\textrm{ mod } \mathscr{I}_n^{K_3}.
\end{align*}

\bibliographystyle{alpha}

\end{document}